\documentclass[11pt]{amsart}
\usepackage[all]{xy}
\usepackage{amscd}
\usepackage{amsfonts}
\usepackage{amsmath}
\usepackage{amssymb}
\usepackage{amsthm}
\usepackage{color}
\usepackage[center]{caption2}
\usepackage{dsfont}
\usepackage{enumerate}
\usepackage{epic}
\usepackage{float}
\usepackage{geometry}
\usepackage{graphicx}
\usepackage{graphics}
\usepackage{hyperref}
\usepackage{indentfirst}
\usepackage{mathrsfs}
\usepackage{multirow}
\usepackage[numbers,sort&compress]{natbib}
\usepackage{pgf}
\allowdisplaybreaks
\numberwithin{equation}{section}
\theoremstyle{plain}
\newtheorem{prop}{Proposition}[section]
\newtheorem{coro}[prop]{Corollary}

\newtheorem{lemm}[prop]{Lemma}

\newtheorem{theorem}[prop]{Theorem}

\theoremstyle{definition}
\newtheorem{defi}[prop]{Definition}

\newtheorem{exam}[prop]{Example}
\newtheorem{rema}[prop]{Remark}

\renewcommand\aa{a}

\newcommand\AAA{\mathcal{A}}
\newcommand\bb{b}
\newcommand\BB{B}
\newcommand\MB{\mathcal{B}}
\newcommand\BC{\mathsf{BC}}
\newcommand\cc{c}
\newcommand\CCC{C}
\newcommand\cont{\mathsf{cont}}
\newcommand\dd{d}

\newcommand\ff{f}
\renewcommand\gg{g}
\newcommand\GK{Gelfand-Kirillov dimension}
\newcommand\gkd[1]{\mathsf{GKdim}(#1)}

\newcommand{\gb}{Gr\"{o}bner basis}
\newcommand\gbs{Gr\"{o}bner bases}
\newcommand{\gsb}{Gr\"{o}bner--Shirshov basis}
\newcommand\gsbs{Gr\"{o}bner--Shirshov bases}
\newcommand\hh{h}
\newcommand\HS[1]{\leavevmode\null\hspace{#1mm}}
\newcommand\id{\mathsf{Id}}
\newcommand\Id[1]{\id(#1)}
\newcommand\ii{i}
\newcounter{ITEM}
\newcommand\ITEM[1]{\setcounter{ITEM}{#1}\leavevmode\hbox{\rm(\roman{ITEM})}}
\newcommand\Irr{\mathsf{Irr}}
\newcommand\jj{j}
\newcommand\kk{k}

\newcommand\LLL{L}
\newcommand\lcm{\mathsf{lcm}}

\newcommand\Lnormed[1]{[#1]{_{_\mathsf{L}}}}

\newcommand\mB{\mathcal{B}}
\newcommand\minn{\mathsf{min}}
\newcommand\mm{m}
\newcommand\MMM{M}
\newcommand\nn{n}
\newcommand\NNN{N}
\newcommand\NF{\mathsf{NF}}
\newcommand\ov[1]{\overline{#1}}
\newcommand\qq{q}
\newcommand\pdots{\mathrel{\HS{0.2}{\cdot}{\cdot}{\cdot}\HS{0.2}}}
\newcommand\pp{p}
\newcommand\rr{r}
\newcommand\RRR{R}

\newcommand\resp{{\it resp.}}
\newcommand\Rnormed[1]{[#1]{_{_\mathsf{R}}}}
\renewcommand\ss{s}
\renewcommand\SS{S}
\newcommand{\sm}{\setminus}
\newcommand\Span{\mathsf{span}}
\newcommand\Sd{\mathsf{Sdeg}}
\renewcommand\tt{t}
\newcommand\TT{T}
\newcommand\uu{u}
\newcommand\UUU{U}
\newcommand\vv{v}
\newcommand\VVV{V}
\newcommand\wdots{, ...\HS{0.2}, }
\newcommand{\wt}{\mathsf{wt}}
\newcommand\ww{w}

\newcommand\xx{x}
\newcommand\XXX{X}
\newcommand\YYY{Y}
\newcommand\yy{y}
\newcommand\zz{z}

\title{Gelfand-Kirillov dimension of bicommutative algebras}

\author{Yuxiu Bai}
\address{Y.B., School of Mathematical Sciences, South China Normal University Guangzhou 510631, P. R. China}
\email{\small 550291315@qq.com}

\author{Yuqun Chen}
\address{Y.C., School of Mathematical Sciences, South China Normal University, Guangzhou 510631, P. R. China}
\email{yqchen@scnu.edu.cn}

\author{Zerui Zhang$^{\sharp}$}
\address{Z.Z., School of Mathematical Sciences, South China Normal University, Guangzhou 510631, P. R. China}
\email{\small zeruizhang@scnu.edu.cn}

\thanks{${}^{\sharp}$ Corresponding author}


\begin{document}
\renewcommand{\thefootnote}{\fnsymbol{footnote}}
\renewcommand{\thefootnote}{\arabic{footnote}}
\begin{abstract}
We first offer a fast method for calculating the Gelfand-Kirillov dimension of a finitely presented commutative algebra by investigating certain finite set.
Then we establish  a \gsbs\ theory for bicommutative algebras,  and
show that every finitely generated bicommutative algebra has a finite \gsb.
As an application, we show that the \GK\ of a finitely generated bicommutative algebra is a nonnegative integer.
\end{abstract}
\maketitle
\section{Introduction}
The Gelfand-Kirillov dimension measures the asymptotic rate of  growth of algebras. And it has become an important tool in the study of algebras. There are many well-known results on the \GK s of algebras. For instance, the Gelfand-Kirillov dimension of a finitely generated commutative algebra~$\AAA$ over a field is the classical Krull dimension of~$\AAA$~\cite[Theorem 4.5]{Bell}, in particular, it is a nonnegative integer.

 There exist a lot of important results concerning the Gelfand-Kirillov dimensions of algebras. Bergman's gap Theorem~\cite{Bergman} shows that there is no algebra with \GK\ strictly between one and two.
Martinez and Zelmanov\cite{Jordan-GK1} studied  finitely generated linear Jordan algebras of \GK\ one and showed that the gap theorem holds for Jordan algebras.
Petrogradsky and Shestakov \cite{Jordan-super-growth} showed that the Gelfand-Kirillov dimension of a Jordan superalgebra can be an arbitrary number in $\{0\}\cup[1,+\infty]$.
The \GK\ of a Leavitt path algebra was investigated by Alahmadi et al. in~\cite{leavitt}.
Qi et al. \cite{nonsymmetric-operad} proved that no finitely generated locally finite nonsymmetric operad has Gelfand-Kirillov dimension strictly between 1 and 2.
And  Bao et al. \cite{symmetric-operad} classified  $\kk$-linear  symmetric operads of low Gelfand-Kirillov dimension.  Petrogradsky, Shestakov, Zelmanov \cite{nil-lie1} and
Shestakov, Zelmanov~\cite{nil-lie2} investigated the Gelfand-Kirillov dimensions of a family of infinite-dimensional nil Lie algebras. Drensky and Papistas~\cite{GKrelative} showed that, under some natural restrictions, the Gelfand-Kirillov dimension of the group of tame automorphisms of relatively free algebras $F_n(V)$ is equal to the
Gelfand-Kirillov dimension of the algebra $F_n(V)$.
Centrone \cite{GKPI} proved a strict relation between the Gelfand-Kirillov dimension of the relatively free (graded) algebra of a PI-algebra and its (graded) exponent.
Zhao and Zhang \cite{zhao2016} offered an algorithm for computing the Gelfand-Kirillov dimension of certain type of differential difference modules.

Recall that
an algebra $\mB$ over a field $\kk$ is called {\em bicommutative} if it satisfies the identities
\begin{eqnarray*}
x(yz)&=&y(xz)\ \ \ \ \mbox{(left commutativity)} \\
(xy)z&=&(xz)y\ \ \ \ \mbox{(right commutativity)}
\end{eqnarray*}
for all $x,y,z\in \mB$.  We denote by~$\BC(\XXX)$ the free bicommutative algebra
generated by~$\XXX$. In~\cite{bicom-basis}, Dzhumadil'daev,  Ismailov and Tulenbaev constructed a linear basis for~$\BC(\XXX)$ when~$\XXX$ is finite, which can be easily generalized to the case when~$\XXX$ is an arbitrary linearly ordered set, see also~\cite{bicom-variety}.
Drensky and Zhakhayev~\cite{bicom-variety} proved that every finitely generated bicommutative algebra is weakly noetherian and offered a positive solution to the finite basis problem for varieties of bicommutative algebras over an arbitrary field of any characteristic.  Classifications on low dimensional bicommutative algebras also attracts peoples' attraction, and it is proved that there are only three  pairwise nonisomorphic two-dimensional bicommutative algebras over an algebraically closed field by Drensky~\cite{classification}.
Kaygorodov,  P\'aez-Guill\'an and Voronin \cite{Ivan1} offered the algebraic and geometric classification of 4-dimensional nilpotent
bicommutative algebras.
Dzhumadil'daev and Ismailov \cite{Jordanelement} established criterions for elements of a free bicommutative algebra to be Lie or Jordan.

We continue to study bicommutative algebras and show that the Gelfand-Kirillov dimension of an arbitrary finitely generated bicommutative algebra is a nonnegative integer.
Moreover, we show that, in order to calculate the Gelfand-Kirillov dimension of a finitely generated bicommutative algebra, it suffices to investigate a finite specific set instead of constructing a possibly infinite linear basis for the bicommutative algebra under consideration.  The method we applied for the proof consists of using Gr\"obner bases and Gr\"obner--Shirshov bases, as introduced by Buchberger~\cite{Buth,Bupaper} and Shirshov~\cite{shirshov1962,shirshov-selectedwork}.
For more history on \gbs\ and \gsbs\ theory, we refer to \cite{survey}.

The paper is organized as follows. In Section~\ref{sec-gk-defi}, we recall the \gbs\ theory for commutative algebras and offer a fast method for calculating  the \GK \ of a finitely generated commutative algebra.  In Section~\ref{gsb}, we establish a \gsbs\ theory for bicommutative algebras and show that every finitely generated bicommutative algebra has a finite \gsb. In Section~\ref{sec-main-result}, we apply the method of \gsbs\ to prove that the \GK\ of a finitely generated bicommutative algebra is a nonnegative integer.  Moreover,  in Theorem~\ref{GKBC} we offer a fast method for calculating the Gelfand-Kirillov dimension of a finitely presented bicommutative algebra~$\mB$ by investigating certain finite set in~$\mB$.

\section{\GK\ of commutative algebras}\label{sec-gk-defi}
Our aim in this section is to offer a new proof for the known fact that the \GK\ of a finitely generated commutative algebra is a nonnegative integer by using the \gbs\ theory for commutative algebras.

Let~$(\XXX,\prec)$ be a fixed well-ordered set and let~$\mathbb{N}$ (\resp\ $\mathbb{N^+}$) be the set of all nonnegative integers (\resp\ all positive integers). Furthermore, we denote by~$[\XXX]$ (\resp\ $[\XXX]^+$) the free commutative monoid (\resp\ free commutative semigroup) generated by~$\XXX$, that is
$$[\XXX]=\{\xx_{i_1}^{\nn_{i_1}}\cdots\xx_{i_t}^{\nn_{i_t}}\mid \xx_{i_1},\dots,\xx_{i_t}\in\XXX,\ \xx_{i_1}\prec\cdots\prec\xx_{i_t},  t,\nn_{i_1}\wdots \nn_{i_t} \in \mathbb{N}\} $$
and $[\XXX]^+=[\XXX]\setminus \{\varepsilon\}$, where~$\xx^0=\varepsilon$ for every~${\xx\in \XXX}$ and~$\varepsilon$ is the unit of the monoid~$[\XXX]$.
When we write a monomial $\uu=\xx_{i_1}^{\nn_{i_1}}\cdots\xx_{i_t}^{\nn_{i_t}}$, we always assume that all the letters~$\xx_{i_1}\wdots\xx_{i_t}$ are pairwise different unless otherwise specified.
Finally, we denote by~$\kk[\XXX]$ the free commutative algebra generated by~$\XXX$, and for an arbitrary set~$\YYY$, we denote by~$\#\YYY$ the cardinality of~$\YYY$.

Now we recall the \gbs\ theory for commutative algebras, for more details, see~\cite{Buth,gsbook1}.  For every monomial~$u=\xx_{i_1}^{\nn_{i_1}}\cdots\xx_{i_t}^{\nn_{i_t}}\in [\XXX]$,  we define the \emph{length} of $\uu$, denoted by~$\ell(\uu)$, to be~$\sum_{j=1}^{t}\nn_{i_j}$.
Let $<$ be a monomial order on $[\XXX]$, that is,  $<$ is a well order and for all monomials~$u,v,w\in[X]$, we have $u<v\Rightarrow wu<wv$.
For every nonzero polynomial~$\ff=\sum \alpha_{i}\uu_i $ in~$\kk[\XXX]$, where each~$0\neq \alpha_i$ lies in~$\kk$, each~$\uu_i$ lies in~$[\XXX]$ and~$\uu_1>\uu_2>\cdots$,
we call~$\uu_1$ the \emph{leading monomial} of~$\ff$ and denote it by~$\bar\ff$.
We call~$\ff$ a \emph{monic polynomial} if~$\alpha_1=1$. And a set $\SS\subseteq\kk[\XXX]$ is called a \emph{monic set} if every polynomial in $\SS$ is monic.   By convention, the empty set~$S$ is a monic set but~$\{0\}$ is not. We denote by~$\Id\SS$ the ideal of~$\kk[\XXX]$ generated by $\SS$ and define~$\overline{S}=\{\overline{f}\mid f\in S\}$.

Let~$I$ be an ideal of~$\kk[\XXX]$. A monic set~$\SS$ is called a \emph{\gb}   (in~$k[\XXX]$) for the quotient algebra~$\kk[\XXX]/I$ if~$\Id\SS=I$ and for every nonzero polynomial~$\ff\in I$,
we have~${\overline {\ff}=\overline{\ss\vv}}$ for some~$\ss\in\SS$, $\vv\in[\XXX]$. By the \gbs\ theory for commuative algebras~\cite{Buth,gsbook1}, if~$S$ is a \gb\ for the algebra~$\kk[\XXX | \SS]:=\kk[\XXX]/\Id\SS$, then the set
$$\Irr(\SS)=[\XXX]\setminus\{\ww\in [\XXX]\mid \ww=\overline{\ss\vv}\ \mbox{for some}\ \ss\in\SS,\vv\in[\XXX]\}$$
forms  a linear basis for the algebra $\kk[\XXX | \SS]$.  In particular, if~$S$ is  empty, then~$S$ is a \gb\ for~$\kk[\XXX | \SS]=\kk[\XXX]$ and  we have~$\Irr(\SS)=[X]$.

Now we recall the definition of  the  \GK s for associative algebras. Let~$\VVV$ be a finite dimensional subspace of an associative algebra $\AAA$ with a linear basis $\{\aa_1\wdots\aa_{\rr}\}$. Define $\VVV^1=\VVV$, and for $n\geq 1$, denote by~$\VVV^{n}$ the subspace of~$\AAA$ spanned by all monomials of the form $\aa_{i_1}\cdots\aa_{i_n}$, where $\aa_{i_1}\wdots\aa_{i_n}\in\{\aa_1\wdots\aa_{\rr}\}$.  Finally, define $\VVV^{\leq n}=\VVV^{1}+\cdots+\VVV^{n}$.

\begin{defi}\cite{book}\label{assgkdefi}
Let~$\AAA$ be an associative algebra over a field~$\kk$. Then the \GK\ of~$\AAA$ is defined to be
$$ \gkd\AAA=\sup\limits_{\VVV}\overline{\lim\limits_{\nn \to \infty}}\log_{\nn}\mathsf{dim}(\VVV^{\leq\nn}),$$
where the supremum is taken over all finite dimensional subspaces of~$\AAA$.
\end{defi}
It is well-known that if~$\AAA$ is an associative algebra generated by a finite set~$\XXX$, then we have~$$\gkd\AAA=\overline{{\lim\limits_{\nn\to \infty}}}\log_{\nn}
\mathsf{dim}(  (\kk\XXX)^{\leq\nn}).$$
In particular,  for a commutative algebra~$\AAA=\kk[\XXX|\SS]$, where $\XXX$ is a finite set and $\SS$ is a \gb\ for $\AAA$, we have
$$\gkd\AAA=\overline{{\lim\limits_{\nn\to \infty}}}\log_{\nn}
\#\{\uu\in\Irr(\SS)\mid \ell(\uu)\leq\nn\}.$$

Obviously, for all commutative algebras,  their Gelfand-Kirillov dimensions have very close relations with their \gbs. Now we have a close look at the case of finitely generated commutative algebras.  It is well-known that every finitely generated commutative algebra has a finite \gb\  by using Hilbert's Theorem, see also~\cite[Sect. 1.6]{gsbook1}.

We recall that for all monomials~$\uu,\vv\in [\XXX]$, the monomial $u$ is called a \emph{multiple} of~$v$, denoted by~$v|u$, if~$\uu=\vv\ww$ in $[\XXX]$ for some monomial~$w\in [\XXX]$.
Let~$\XXX=\{\xx_1,\dots,\xx_\rr\}$ be a nonempty finite set and let~$\SS$ be a  finite \gb\ for~$k[X|S]$.
Define
\begin{eqnarray*}
\CCC_{\SS}=
\begin{cases}
0,& \mbox{if } \SS=\emptyset\\
\min\{\nn\in \mathbb{N}\mid \bar{\ss}|\xx_1^n\cdots\xx_{\rr}^n,\ss\in \SS\},& \mbox{if } \SS\neq\emptyset.
\end{cases}
\end{eqnarray*}
 In particular, if~$S=\{\varepsilon\}$, then we have $\CCC_{\SS}=0$.  For every $u=\xx_{1}^{\nn_{1}}\cdots\xx_{\rr}^{\nn_{\rr}}\in[\XXX]$, define $$\Sd(\uu)=\#\{i\mid \nn_{i}> \CCC_S, i\in\{1\wdots r\}\}. $$
It follows immediately that we have
$$\Sd(\varepsilon)=0.$$
Finally, for an arbitrary nonempty subset~$B\subseteq \Irr(\SS)$, we define
$$\Sd(B)=\max\{\Sd(\mu)\mid \mu\in B\}$$
and define
$$\Sd(\emptyset)=0.$$
We shall soon see that~$\Sd(B)$ is always finite if~$X$ is a finite set.

Now we offer a characterization for the Gelfand-Kirillov dimension of a finitely generated commutative algebra. In particular, as far as the authors know, it offers a new alternative proof for the fact that the Gelfand-Kirillov dimension of a finitely generated commutative algebra is a nonnegative integer.
\begin{lemm}\label{lemma-ca}
Let~$\XXX$ be a nonempty finite set and let~$S$ be a finite Gr\"obner  basis for~$k[\XXX|\SS]$. Then we have~$\gkd{k[\XXX|\SS]}=\Sd(\Irr(\SS))$.
\end{lemm}
\begin{proof}
Assume that~$\XXX=\{\xx_1,\dots,\xx_\rr\}$. Define $\mm=\Sd(\Irr(\SS))$.
 If~$\mm=0$, then we have~$\Irr(\SS)\subseteq \{\xx_{1}^{\nn_{1}}\cdots\xx_{\rr}^{\nn_{\rr}}\mid \nn_i\leq \CCC_S, i\in\{1\wdots r\}\}$. It follows that~$\gkd{k[\XXX|S]}=0$.

Now we suppose~$\mm\geq 1$. Since~$\mm=\Sd(\Irr(\SS))$, for every~$u\in \Irr(\SS)\subseteq [X]$, we obtain that~$\Sd(\uu)\leq m$.  Therefore, we have
$$\Irr(\SS)\subseteq\{\uu\in[\XXX]\mid \Sd(\uu)\leq\mm\}.$$
Moreover, there exists $\uu\in\Irr(\SS)$ such that $\Sd(\uu)=\mm$. Without loss of generality, we assume $\uu=\xx_{1}^{\nn_{1}}\cdots\xx_{\rr}^{\nn_{\rr}}$, where each $\nn_i\in \mathbb{N}$ and $\nn_{1},\dots,\nn_{\mm}> \CCC_{\SS}$.
For every~$w\in \Irr(\SS)$ and for every monomial~$v\in [X]$ such that~$v|w$, it is clear that~$v$ also lies in~$\Irr(\SS)$. So we deduce that $$\xx_{1}^{\CCC_{\SS}+1}\cdots\xx_{\mm}^{\CCC_{\SS}+1}\in\Irr(\SS).$$
Now we show that
$$\Irr(\SS)\supseteq   \{\xx_{1}^{\pp_{1}}\cdots\xx_{\mm}^{\pp_{\mm}}\mid \pp_{1},\dots,\pp_{\mm}\in\mathbb{N}\}\sm\{\varepsilon\}.
$$
It suffices to show that $\xx_{1}^{\nn}\cdots\xx_{\mm}^{\nn}$ belongs to $\Irr(\SS)$ for every $\nn>\CCC_{\SS}$.  Suppose for the contrary that
${\xx_{1}^{\nn}\cdots\xx_{\mm}^{\nn}\notin\Irr(\SS)}$ for some integer~$\nn>\CCC_{\SS}+1$.
Then we define
$$\nn_0=\minn\{\nn\mid \xx_{1}^{\nn}\cdots\xx_{\mm}^{\nn}\notin\Irr(\SS)\}.$$
Clearly, we have~$$\nn_0>\CCC_{\SS}+1$$ and  $\xx_{1}^{\nn_0-1}\cdots\xx_{\mm}^{\nn_0-1}\in\Irr(\SS)$.
Therefore,  there exists some element~$s\in S$ such
that we have~$\bar\ss|\xx_{1}^{\nn_0}\cdots\xx_{\mm}^{\nn_0}$ while $\bar\ss\nmid\xx_{1}^{\nn_0-1}\cdots\xx_{\mm}^{\nn_0-1}$.
By the definition of~$\CCC_S$, we deduce~$\CCC_S\geq n_0$, which contradicts with the fact that~$\nn_0>\CCC_{\SS}+1$.

Therefore we have
$$\{\xx_{1}^{\pp_{1}}\cdots\xx_{\mm}^{\pp_{\mm}}\mid \pp_{1},\dots,\pp_{\mm}\in\mathbb{N}\}\sm\{\varepsilon\}\subseteq\Irr(\SS)\subseteq
\{\uu\in[\XXX]\mid \Sd(\uu)\leq\mm\}.$$
By straightforward calculations, we obtain~$\gkd{k[\XXX|S]}=\mm$ immediately.
\end{proof}

\begin{theorem}\label{thm-ca}
Let~$\XXX=\{\xx_1,\dots,\xx_{\rr}\}$ be a nonempty finite set and~$S$ a finite Gr\"obner  basis for~$k[\XXX|\SS]$. Then we have
$$\gkd{k[\XXX|\SS]}=\Sd(\{\uu\mid \uu=\xx_{1}^{\nn_{1}}\cdots\xx_{\rr}^{\nn_{\rr}}\in\Irr(\SS),    \nn_i\in \{0, \CCC_S+1\}, i\in\{1\wdots r\}\}).$$
\end{theorem}
\begin{proof}
Define~$$\mm'=\Sd(\{\uu\mid \uu=\xx_{1}^{\nn_{1}}\cdots\xx_{\rr}^{\nn_{\rr}}\in\Irr(\SS),    \nn_i\in \{0, \CCC_S+1\}, i\in\{1\wdots r\}\})$$
and
$$\mm=\Sd(\Irr(\SS)).$$
Clearly we have~$\mm'\leq\mm$. By Lemma~\ref{lemma-ca},  it suffices to show~$m\leq m'$.

If $\mm'=0$, then $\{\uu\mid \uu=\xx_{1}^{\nn_{1}}\cdots\xx_{\rr}^{\nn_{\rr}}\in\Irr(\SS)\sm\{\varepsilon\},    \nn_i\in \{0, \CCC_S+1\}, 1\leq i\leq r\}$ is an empty set. It follows that we have~$\Irr(\SS)\subseteq \{\xx_{1}^{\nn_{1}}\cdots\xx_{\rr}^{\nn_{\rr}}\mid \nn_i\leq \CCC_S, 1\leq i\leq r\}$.  Thus we have~$\mm=0=\mm'$.

Suppose $0<\mm'$. Then we have~$m\geq m'>0$.  So there exists a monomial~$\uu\in\Irr(\SS)$ with~$\Sd(\uu)=\mm>0$. Thus we may assume~$$\uu=\yy_1^{\tt_1}\cdots\yy_{m}^{\tt_{m}}\uu',$$
where~$\uu'\in[\XXX]$, $\Sd(\uu')=0$, $\tt_1\wdots\tt_{m}\geq\CCC_{\SS}+1$, $y_1\wdots y_{m}\in X$ and~$y_1<\cdots < y_{m}$.
Then~$y_{1}^{\CCC_{\SS}+1}\cdots y_{m}^{\CCC_{\SS}+1}$ also lies in $\Irr(\SS)$, in other words,  we have~$m'\geq \Sd(u)=\mm$.

Therefore, by Lemma~\ref{lemma-ca}, we obtain
$$\gkd{k[\XXX|S]}=m=\Sd(\{\uu\mid \uu=\xx_{1}^{\nn_{1}}\cdots\xx_{\rr}^{\nn_{\rr}}
\in\Irr(\SS),    \nn_i\in \{0, \CCC_S+1\}, i\in\{1\wdots r\}\}).$$
The proof is completed.
\end{proof}
 Theorem~\ref{thm-ca} offers a fast method for calculating the Gelfand-Kirillov dimension of a finitely presented commutative algebra by investigating the finite set
$$\{\uu\in\Irr(\SS)\mid \uu=\xx_{1}^{\nn_{1}}\cdots\xx_{\rr}^{\nn_{\rr}},   \nn_i\in \{0, \CCC_S+1\}, i\in\{1\wdots r\}\}$$
 rather than~$\Irr(\SS)$.
 As a special case, if~$X=\{\xx_1,\dots,\xx_{\rr}\}$ and~$S=\emptyset$, then we have~$\CCC_S=0$, $x_1^1\cdots x_r^1\in \Irr(S)$, and thus~$\gkd{\kk[X]}=r$.

\begin{rema}
Theorem~\ref{thm-ca} offers an algorithm for computing the Gelfand-Kirillov dimension (or equivalently, Krull dimension~\cite[Proposition 5.1]{Bell}) of a finitely presented commutative algebra. More precisely, let $\XXX=\{\xx_1,\dots,\xx_{\rr}\}$ be a nonempty finite
set and let~$T$ be a finite subset of~$k[X]$. Then by applying Buchberger's algorithm~\cite[Algotithm 1.7.1]{gsbook1}, we can find a finite \gb\ $\SS$ for the algebra~$\AAA:=k[X|T]$. Then we obtain the nonnegative integer~$\CCC_{\SS}$
and the set~$\BB:=\{\uu\mid \uu=\xx_{1}^{\nn_{1}}\cdots\xx_{\rr}^{\nn_{\rr}}\in\Irr(\SS),    \nn_i\in \{0, \CCC_S+1\}, i\in\{1\wdots r\}\}$.
Since~$B$ and~$S$ are finite sets, there exists an algorithm for calculating~$\Sd(B)$, which is exactly~$\gkd{\AAA}$.
\end{rema}
\begin{exam}
Let $\XXX=\{\xx_1\wdots\xx_6\}$ and~$\SS=\{\xx_1^2\xx_2^3,\xx_3\xx_4\}$.
Then~$\CCC_\SS=3$.
For every subset~$\{\yy_1\wdots\yy_5\}\subseteq \XXX$, we have~$\yy_1^4\cdots\yy_5^4\notin\Irr(\SS)$. Since~$\xx_2^4\xx_3^4\xx_5^4\xx_6^4$ lies in~$\Irr(\SS)$.
Thus we obtain~$\gkd{k[\XXX|\SS]}=4$.
\end{exam}

\begin{coro}\label{porp-1}
Let $\SS$ be a nonempty subset of~$[\XXX]^+$ and~$\uu\in[\XXX]^+$. Suppose that $\gkd{k[\XXX|\SS]}=p\geq1$. If there exists a subset~$\{y_1\wdots y_p\}\subseteq\XXX$ such that some monomial in~$\Irr(S)$ of the form~$\yy_{1}^{\CCC_S+1}\cdots\yy_{p}^{\CCC_S+1}$  lies in~$\Irr(\SS\cup\{\uu\})$,
where $\yy_i\neq\yy_j$ for all distinct~$i,j\leq p$,
then we have
$$\gkd{k[\XXX|\SS\cup\{\uu\}]}=\gkd{k[\XXX|\SS]};$$ Otherwise,  we have $$\gkd{k[\XXX|\SS\cup\{\uu\}]}=\gkd{k[\XXX|\SS]}-1.$$
\end{coro}

\begin{proof}
Obviously, every nonempty subset of $[\XXX]^+$ is a \gb\ in $\kk[\XXX]$, and thus~$S$ is a \gb\ for~$k[\XXX|\SS]$.
 By assumption, we have
$$\gkd{k[\XXX|\SS]}=\Sd(\{\xx_{1}^{\nn_{1}}\cdots\xx_{\rr}^{\nn_{\rr}}
\in\Irr(\SS)\mid \nn_i\in \{0, \CCC_S+1\}, 1\leq i\leq r\})=p, \mbox{ where } p\geq1.$$
If there exists a subset~$\{y_1\wdots y_p\}\subseteq\XXX$, such that some monomial in~$\Irr(S)$ of the form~$\yy_{1}^{\CCC_S+1}\cdots\yy_{p}^{\CCC_S+1}$  lies in~$\Irr(\SS\cup\{\uu\})$,
where $\yy_i\neq\yy_j$ for all distinct~$i,j\leq p$,
then we have
$$\gkd{k[\XXX|\SS\cup\{\uu\}]}=p=\gkd{k[\XXX|\SS]}.$$

Now we assume that for every subset~$\{y_1\wdots y_p\}\subseteq\XXX$, every monomial in~$\Irr(S)$ of the form~$\yy_{1}^{\CCC_S+1}\cdots\yy_{p}^{\CCC_S+1}$ does not belong to the set~$\Irr(\SS\cup\{\uu\})$,
then we have~$\uu|\yy_{1}^{\CCC_S+1}\cdots\yy_{p}^{\CCC_S+1}$.
If~$p=1$, then~$\{\xx_{1}^{\nn_{1}}\cdots\xx_{\rr}^{\nn_{\rr}}\in\Irr(\SS\cup\{\uu\})\mid \nn_i\in \{0, \CCC_S+1\}, i\in\{1\wdots r\}\}=\{\varepsilon\}$, and thus
$$\gkd{k[\XXX|\SS\cup\{\uu\}]}=0=p-1=\gkd{k[\XXX|\SS]}-1.$$
If~$p>1$, then without loss of generality, we may assume~$\xx_{i_1}^{\CCC_S+1}\cdots\xx_{i_p}^{\CCC_S+1}\in \Irr(\SS)$
and~$\xx_{i_1}|\uu$.  It follows that~$\xx_{i_2}^{\CCC_S+1}\cdots\xx_{i_p}^{\CCC_S+1}$
lies in~$\Irr(\SS\cup\{\uu\})$. Therefore we obtain that
$$\gkd{k[\XXX|\SS\cup\{\uu\}]}=p-1=\gkd{k[\XXX|\SS]}-1.$$
The proof is completed.
\end{proof}

The following proposition offers a rough estimation of the Gelfand-Kirillov dimension of a finitely presented commutative algebra.
\begin{prop}\label{gkc}
Let~$\XXX=\{\xx_1,\dots,\xx_{\rr}\}$ be a nonempty finite set and let~$S=\{\ff_1\wdots\ff_t\}$ be a finite Gr\"obner  basis for~$k[\XXX|\SS]$.
Then the following hold:

\ITEM1 If~$1\leq t\leq r$ and $\varepsilon\notin\SS$, then we have
$$\gkd{k[\XXX|\SS]}\in\{ r-t\wdots r-1\};$$

\ITEM2 If~$t>r$, then we have
$$\gkd{k[\XXX|\SS]}\in \{0\wdots r-1\}.$$
Moreover, if $\rr>1$ and~$\tt<\rr$, then~$\gkd{k[\XXX|\SS]}=0\Leftrightarrow \varepsilon\in\SS$.
\end{prop}

\begin{proof}
\ITEM1  We use induction on~$t$.
If $t=1$, then we may assume that~$\bar\ff_1=\xx_{1}^{\nn_{1}}\cdots\xx_{\rr}^{\nn_{\rr}}$, where~$\nn_i\in \mathbb{N}$ and~$\nn_1>0$.
Thus we have~$\CCC_S=\max\{\nn_{1}\wdots \nn_{\rr}\}$. Clearly we have
$$x_2^{\CCC_S+1}\cdots x_r^{\CCC_S+1}\in \Irr(S) \mbox{ and } x_1^{\CCC_S+1}x_2^{\CCC_S+1}\cdots x_r^{\CCC_S+1}\notin \Irr(S). $$
  Therefore by Theorem~\ref{thm-ca}, we
obtain
$$\gkd{k[\XXX|\SS]}=\Sd(\{\uu\mid \uu=\xx_{1}^{\nn_{1}}\cdots\xx_{\rr}^{\nn_{\rr}}\in\Irr(\SS),    \nn_i\in \{0, \CCC_S+1\}, i\in\{1\wdots r\}\})=r-1.$$
Now we assume~$t\geq 2$.
Obviously, $\{\ov{\ff_1}\wdots\ov{\ff_{t-1}}\}$ is a \gb\ in $\kk[\XXX]$ and~$\Irr(\SS)=\Irr(\ov{\SS})$.
By induction hypothesis, we have $r-t+1\geq1$ and $$p:=\gkd{k[\XXX|\{\ov{\ff_1}\wdots\ov{\ff_{t-1}}\}]}\in\{r-t+1\wdots r-1\}.$$
By Corollary~\ref{porp-1}, we have
$$\gkd{k[\XXX|\ov{\SS}]}=\gkd{k[\XXX|\SS]}\in \{p,p-1\}\subseteq \{r-t\wdots r-1\}.$$

\ITEM2
Since~$x_1^{\CCC_S+1}x_2^{\CCC_S+1}\cdots x_r^{\CCC_S+1}\notin \Irr(S)$, by Theorem \ref{thm-ca}, the result follows.
\end{proof}

\begin{exam} Let~$\XXX=\{\xx_1,\dots,\xx_{\rr}\}$ be a nonempty finite set.
Given a positive integer~$t$, for every integer~$\pp$ such that~$1\leq p\leq \min\{t,r\}$, we define $m=t-p$ and define
$$\SS=\{\xx_i^{4m+1}\mid1\leq i\leq p\}\cup\{\xx_{1}^j\xx_r^{3m-j}\mid m\leq j\leq 2m-1\}.$$
Note that if~$t=p$, then we have~$\SS=\{\xx_i\mid1\leq i\leq t\}$.
Clearly~$\SS$ is a \gb\ for~$k[\XXX|\SS]$ with~$\#\SS=t$ and we have~$C_S=4m+1$. Let
$$B:=\{\uu\in[\XXX]\mid \uu=\xx_{1}^{\nn_{1}}\cdots\xx_{\rr}^{\nn_{\rr}},   \nn_i\in \{0, \CCC_S+1\}, i\in\{1\wdots r\}\}.$$
 If~$p<r$, then we have
$$\xx_{p+1}^{\CCC_S+1}\cdots\xx_{r}^{\CCC_S+1}\in B\cap\Irr(\SS),$$
and for every monomial~$\uu\in B$ with $\Sd(\mu)=r-p+1$, $\uu$ does not lie in~$\Irr(\SS)$.
Thus we obtain
$$\gkd{k[\XXX|\SS]}=r-p\in\{r-t\wdots r-1\}.$$
If~$p=r$, then we have
$B\cap\Irr(\SS)=\{\varepsilon\}.$
Thus
$\gkd{k[\XXX|\SS]}=r-p=0.$

It follows that if~$S$ is a \gb\ consisting of~$t\geq 1$ elements and~$\varepsilon\notin\SS$, then the possible Gelfand-Kirillov dimension of the algebra~$k[X|S]$ can be an arbitrary integer between~$\max\{r-t,0\}$ and~$r-1$.

\end{exam}

\begin{rema}
In Proposition~\ref{gkc},  a \gb~$S$ with smaller cardinality might improve the estimation.   So in  Proposition~\ref{gkc}, we may assume~$S$ to be a minimal \gb. Recall that a \gb~$S$ is called \emph{minimal} if for all~$f,g\in S$, $\bar\ff$ is not a multiple of~$\bar\gg$. Moreover,  when the order is fixed, the cardinality of a minimal \gb\ is also minimum among all the possible cardinalities of \gbs\ for~$k[X|S]$.
If we use another monomial order~$\prec$ on~$[\XXX]$, then we can find a minimal \gb~$T$ for~$\kk[\XXX|S]$  with respect to the order~$\prec$ and probably we have~$\#T\neq \#S$.
\end{rema}

 \section{\gsbs\ theory for bicommutative algebras}\label{gsb}
Our aim in this section is to
establish  a \gsbs\ theory for bicommutative algebras and show that every finitely generated bicommutative algebra has a finite \gsb.
\subsection{Composition-Diamond lemma for bicommutative algebras}\label{gsbforbi}
We first recall that \emph{terms} over~$\XXX$ are ``words with brackets''; they are
 defined by the following induction:

\ITEM1 Every element of~$\XXX$ is a term over~$\XXX$;

\ITEM2 If~$\mu, \nu$ are terms over~$\XXX$, then so is~$(\mu\nu)$.

\noindent We denote by~$\XXX^{(*)}$ the set of all terms over~$\XXX$.

Let~$\kk\XXX^{(*)}$ be the free nonassociative algebra generated by~$\XXX$. Then clearly $\XXX^{(*)}$ forms a linear basis for~$\kk\XXX^{(*)}$.
For every~$\mu$ in~$\XXX^{(*)}$, we define the \emph{content}~$\cont(\mu)$ of~$\mu$ to be the set of all letters that appear in $\mu$ and define the \emph{length}~$\ell(\mu)$ of~$\mu$ to be the number of letters (with repetition) that appear in $\mu$. For instance, for $\mu=((\xx_1\xx_1)\xx_2)\in \XXX^{(\ast)}$, we have~$\cont(\mu)=\{\xx_1,\xx_2\}$ and~$\ell(\mu)=3$.

For all~$\aa_1\wdots\aa_n$ ($\nn\geq 1$) in a bicommutative algebra~$\mB$, we define
$$
\Lnormed{\aa_1\wdots\aa_\nn}=(\cdots (((\aa_1 \aa_2)  \aa_3)   \cdots \aa_{\nn-1})  \aa_\nn) \ \ \mbox{(left-normed bracketing),}
$$
and define
$$
\Rnormed{\aa_1\wdots\aa_\nn}=(\aa_1 (\aa_2 \cdots   (\aa_{\nn-2}  (\aa_{\nn-1}  \aa_\nn))\cdots)) \ \ \mbox{(right-normed bracketing).}
$$
In particular, we have~$[\aa_1]_\mathsf{_L}=[\aa_1]_\mathsf{_R}=\aa_1$ and~$[\aa_1,\aa_2]_\mathsf{_L}=[\aa_1,\aa_2]_\mathsf{_R}=
(\aa_1 \aa_2)$.
For all elements~$\aa_1\wdots\aa_\pp,\bb_1\wdots\bb_\qq\in\mB$, where~$\pp,\qq\in \mathbb{N}^+$, we define
$$
[\aa_1\cdots\aa_\pp;\bb_1\cdots\bb_\qq]
=\Rnormed{\aa_1\wdots\aa_{\pp-1}, \Lnormed{\aa_\pp,\bb_1\wdots\bb_\qq}}.
$$

Now we recall some known identities in bicommutative algebras established in~\cite{bicom2003,bicom-basis}, see also~\cite{sz}.
\begin{lemm}\emph{\cite{bicom2003,bicom-basis}}\label{lemm-ass}
  For all~$\aa_1\wdots\aa_\pp,\bb_1\wdots\bb_\qq,
  \cc_1\wdots\cc_\mm,\dd_1\wdots\dd_\nn\in\mB$, for all permutations~$\sigma\in S_\pp, \delta\in\SS_\qq$, where~$\pp,\qq,\mm,\nn\in \mathbb{N}^+$, and $\SS_p, \SS_q$ are symmetric groups of orders $p,q$ respectively, we have

\ITEM1 ~$((\aa_1(\aa_2\aa_3))\aa_4)
=(\aa_1((\aa_2\aa_3)\aa_4))$,

  \ITEM2 $[\aa_1\cdots\aa_\pp;\bb_1\cdots\bb_\qq]=\Rnormed{\aa_1\wdots\aa_{\pp-1}, \Lnormed{\aa_\pp,\bb_1\wdots\bb_\qq}}
=\Lnormed{{\Rnormed{\aa_1\wdots\aa_\pp,\bb_1},\bb_2\wdots\bb_\qq}}$. In particular, we have~$[\aa_1\cdots\aa_\pp;\bb_1\cdots\bb_\qq]
=[\aa_{\sigma(1)}\cdots\aa_{\sigma(\pp)};
\bb_{\delta(1)}\cdots\bb_{\delta(\qq)}]$.

\ITEM3 $([\aa_1\cdots\aa_\pp;\bb_1\cdots\bb_\qq]
[\cc_1\cdots\cc_\mm;\dd_1\cdots\dd_\nn])
=[\aa_1\cdots\aa_\pp\cc_1\cdots\cc_\mm;
\bb_1\cdots\bb_\qq\dd_1\cdots\dd_\nn]$.
\end{lemm}
Let~$\BC(\XXX)$ be the free bicommutative algebra
generated by a well-ordered set~$\XXX$.
Note that by Lemma~\ref{lemm-ass}, the ordering of~$\aa_\ii$ and that of~$\bb_\jj$ do not make difference in the monomial~$[\aa_1\cdots\aa_\pp;\bb_1\cdots\bb_\qq]$, and what is important is whether they are on the ``left hand side" or on the ``right hand side".
In particular, for all~$\pp,\qq\in \mathbb{N}^{+}$, for all  letters~$\zz_1\wdots\zz_\pp,\yy_1\wdots\yy_\qq\in \XXX$, when we write an element of the form~$[\zz_1\cdots\zz_\pp;\yy_1\cdots\yy_\qq]\in\BC(\XXX)$, it is not necessary to assume~$\zz_1\leq \cdots\leq \zz_\pp$ or~$\yy_1\leq \cdots\leq \yy_\qq$.

Moreover, for all~$\uu=\zz_1\cdots\zz_\pp,\vv=\yy_1\cdots\yy_\qq\in[\XXX]^+$, we define $$[\uu;\vv]=[\zz_1\cdots\zz_\pp;\yy_1\cdots\yy_\qq],$$
 and
  $$\NF(\XXX)=\{[\uu;\vv]\mid \uu,\vv\in[\XXX]^+\}\cup\XXX.$$
Then by~\cite{bicom-basis},  $\NF(\XXX)$ forms a linear basis for~$\BC(\XXX)$.
For the convenience of the readers,  we list the multiplication table below. For all~$\zz,\yy\in \XXX$ and $[\uu;\vv],
[\uu';\vv']\in \NF(\XXX)\sm\XXX$, we have
\begin{equation}\label{multiplication}
(\zz\yy)=[\zz;\yy];\
(\zz[\uu;\vv])
=[\zz\uu;\vv];\
([\uu;\vv]\zz)
=[\uu;\vv\zz];\
([\uu;\vv]
[\uu';\vv'])=
[\uu\uu';
\vv\vv'].
\end{equation}

In light of the multiplication table~\eqref{multiplication} above, we can generalize the notion of multiple in~$\NF(X)$.  More precisely, for all~$\mu, \nu\in \NF(X)$, we call~$\nu$ a \emph{multiple} of~$\mu$, denoted by~$\mu|\nu$, if one of the followings holds:
 \ITEM1  $\mu=\xx\in\XXX$ and $\xx\in\cont(\nu)$;
 \ITEM2 $\mu=[\uu;\vv]$, $\nu=[\uu';\vv']\in\NF(\XXX)$ for some~$\uu,\vv,\uu',\vv'\in[\XXX]^+$,  and we have~$\uu |\uu'$, $\vv |\vv'$ in $[\XXX]^+$.

Let $<$ be a well order on~$\NF(\XXX)$.
For every nonzero polynomial~$\ff=\sum \alpha_{i}\mu_i $ in~$\BC(\XXX)$, where each~$0\neq \alpha_i$ lies in~$\kk$, each~$\mu_i$ lies in~$\NF(\XXX)$ and~$\mu_1>\mu_2>\cdots$,
we call~$\mu_1$ the \emph{leading monomial} of~$\ff$, denoted by~$\bar\ff$; and call~$\alpha_1$ the \emph{leading coefficient} of~$\ff$,  denoted by~$\mathsf{lc}(\ff)$.
We call~$\ff$ a \emph{monic polynomial} if~$\mathsf{lc}(\ff)=1$.  Finally,  a set $\SS\subseteq\BC(\XXX)$ is called a \emph{monic set} if every polynomial in~$\SS$  is monic. By convention, the empty set~$S$ is a monic set but~$\{0\}$ is not.

A well order $<$ on~$\NF(\XXX)$ is called a \emph{monomial order} if  for all monomials~$\mu,\nu,\nu'\in\NF(\XXX)$, we have $\mu<\nu\Rightarrow \overline{(\mu\nu')}<\overline{(\nu\nu')} \mbox{ and }\overline{(\nu'\mu)}<\overline{(\nu'\nu)}$.
 Here we offer two specific examples of monomial order on~$\NF(\XXX)$.

\begin{exam}
Let~$\prec$ be a monomial order on $[\XXX]^+$.
For every $\mu\in \NF(\XXX)$, we define
\begin{displaymath}
\wt(\mu)=
\begin{cases}
(1,1,\xx), &\text{if}\ \mu=\xx\in\XXX,\\
(\ell(u)+\ell(v),\ell(u),\uu,\vv) &\text{if}\ \mu=[\uu;\vv]\in\NF(\XXX)\sm\XXX.
\end{cases}
\end{displaymath}
For all $\mu,\nu\in\NF(\XXX)$, we define
$$\mu<\nu \text{ if and only if } \wt(\mu)<\wt(\nu) \text{ lexicographically},
$$
where the elements in~$[\XXX]^+$ are compared by~$\prec$. Then clearly~$<$ is a monomial order on~$\NF(\XXX)$.
\end{exam}

\begin{exam}
Let $\prec$ be a monomial order on $[\XXX]^+$. For every $\mu\in \NF(\XXX)$, we define
\begin{displaymath}
\wt(\mu)=
\begin{cases}
(\xx,\xx,\xx), &\text{if}\ \mu=\xx\in\XXX,\\
(\uu\vv,\uu,\vv) &\text{if}\ \mu=[\uu;\vv]\in\NF(\XXX)\sm\XXX,
\end{cases}
\end{displaymath}
where $\uu\vv$ is a commutative word in $[\XXX]^+$.
For all $\mu,\nu\in\NF(\XXX)$, we define
$$\mu<\nu \text{ if and only if } \wt(\mu)<\wt(\nu) \text{ lexicographically},
$$
where the elements in~$[\XXX]^+$ are compared by~$\prec$. Then clearly~$<$ is a monomial order on~$\NF(\XXX)$.
\end{exam}

From now on, let~$<$ be a fixed monomial order on~$\NF(X)$ and let~$\SS$ be a monic subset of~$\BC(\XXX)$.  Denote by~$\Id\SS$  the ideal of $\BC(\XXX)$ generated by~$\SS$. For every~$s\in S$, we call~$s$ an~\emph{$S$-polynomial}, and for every~$S$-polynomial~$h$, for every~$\mu$ in~$\NF(\XXX)$, both~$(\mu\hh)$ and~$(\hh\mu)$ are called~\emph{$S$-polynomials}.
Obviously every element in~$\Id\SS$ can be written as a linear combination of $S$-polynomials.
Among~$S$-polynomials, we  introduce some special polynomials in~$\Id\SS$, which will form a linear generating set of~$\Id\SS$ under some conditions.

\begin{defi}
Let  $\SS$ be a monic set in $\BC(\XXX)$.  For every~$\ss\in\SS$,\\
\ITEM1  the polynomial~$\ss$ is called a \emph{normal~$\SS$-polynomial};\\
\ITEM2 if~$\ell(\bar\ss)\geq 2$, then for all~$\zz_1\wdots\zz_p,\yy_1\wdots\yy_q\in\XXX$, $\pp,\qq\in\mathbb{N}$, we call~$\Rnormed{\zz_1\wdots\zz_p,\Lnormed{\ss,\yy_1\wdots\yy_q}}$ a \emph{normal~$\SS$-polynomial}.
\end{defi}

The following lemma  shows that the set:
$$
\Irr(S):=\NF(\XXX)\sm\{\mu\in \NF(\XXX)\mid \mu= \overline{\hh}, \hh \mbox{ is a normal }\SS\mbox{-polynomial} \}
$$
is a linear generating set of the bicommutative algebra~$\BC(\XXX|\SS):=\BC(\XXX)/\Id\SS$.
\begin{lemm}\label{f=Irr+n-s-polynomials}
 For every~$\ff\in \BC(\XXX)$, we have
$$\ff=\sum_{\mu_i\leq \bar f }\alpha_i\mu_i+
\sum_{\ov{h_{j}}\leq\bar \ff}\beta_jh_{j},
$$
where each $\alpha_i,\beta_j\in \kk$, $ \mu_i\in \Irr(\SS)$ and $\hh_{j}$ is
a normal $\SS$-polynomial.  In particular, it follows that   $\Irr(\SS)$ is a set of linear generators of the
algebra $ \BC(\XXX|\SS)$.
\end{lemm}
\begin{proof}
The result follows by induction on $\bar \ff $.
\end{proof}
\begin{defi}\label{defi-gsb}
  Let~$I$ be an ideal of~$\BC(\XXX)$. A  monic set~$S$ is called a \emph{Gr\"obner-Shirshov basis} (in~$\BC(\XXX)$) for~$\BC(\XXX)/I$, if~$\Id\SS=I$ and for every nonzero polynomial~$\ff\in \Id\SS$, we have~$\overline {\ff}=\overline{\hh}$ for some normal~$\SS$-polynomial~$\hh$.
\end{defi}
We shall see that, if~$\SS$ is a \gsb\ for~$\BC(\XXX|\SS)$, then~$\Irr(\SS)$ forms a linear basis for~$\BC(\XXX|\SS)$. In particular, if~$S$ is an empty set, then~$S$ is a \gsb\ for~$\BC(\XXX|\SS)$ and  we have~$\Irr(\SS)=\NF(\XXX)$.  But the conditions in Definition~\ref{defi-gsb} is not tractable in practice. Therefore,  we shall introduce certain special polynomials (hereafter called compositions) in~$\Id\SS$, such that if these polynomials satisfy some easily managed conditions, then~$S$ forms a \gsb.

Before going there, we introduce the notion of least common multiple in~$\NF(\XXX)\sm \XXX$.
For all $\uu,\vv\in[\XXX]$, we denote by $\lcm(\uu,\vv)$  the least common multiple of~$u$ and~$v$ in~$[\XXX]$.
 For all~$\mu=[\uu;\vv],\nu=[\uu';\vv']\in\NF(\XXX)$, where~$\uu,\vv,\uu',\vv'$ lie in $[\XXX]^+$, we call $[\lcm(\uu,\uu');\lcm(\vv,\vv')]$
the \emph{least common multiple} of~$\mu$ and~$\nu$, and denote it by~$\lcm(\mu,\nu)$.

\begin{defi}\label{defi-gsb}
Let $\SS$ be a monic subset of~$ \BC(\XXX)$. For all  $\ff,\gg\in \SS$, we define \emph{compositions} of~$S$ as follows:

\ITEM1 If $\ell(\ov{\ff})= 1$, then for every $\zz\in\XXX$, we call $(\zz\ff)$ a \emph{left multiplication composition} and~$(\ff\zz)$ a \emph{right multiplication composition};

\ITEM2 If $\overline{\ff}=\overline{\gg}\in\XXX$, then we call $(\ff,\gg)_{\overline{\ff}}=\ff-\gg$ an \emph{inclusion composition};

\ITEM3 If $\ell(\ov{\ff})\geq 2$ and $\ff\notin \Span_\kk(\NF(\XXX)\sm \XXX)$, where~$\Span_\kk(\NF(\XXX)\sm \XXX)$ is the subspace of~$\BC(\XXX)$ spanned by~$\Span_\kk(\NF(\XXX)\sm \XXX)$, then for all $\yy,\zz\in\XXX$, we call
$$(f,f)_{\ov{(\yy(\ff\zz))}}=(\yy(\ff\zz))-((\yy\ff)\zz)$$ a \emph{multiplication composition};

\ITEM4
If $\bar\ff=[\uu,\vv],\bar\gg=[\uu',\vv']\in\NF(\XXX)\sm\XXX$,
and $\ell(\lcm(\bar\ff,\bar\gg))<
\ell(\bar\ff)
+\ell(\bar\gg)$, where~$\uu,\vv,\uu',\vv'\in[\XXX]^+$, $\lcm(\uu,\uu')=\uu\aa_1\cdots\aa_{p}
=\uu'\bb_1\cdots\bb_{m}$,
$\lcm(\vv,\vv')=\vv\cc_1\cdots\cc_{q}=\vv'\dd_1\cdots\dd_{n}$,
and~$
\aa_1\wdots\aa_{p}$,
$\bb_1\wdots\bb_{m}$,
$\cc_1\wdots\cc_{q}$,
$\dd_1\wdots\dd_{n}
\in\XXX$, $p,q,m,n\in\mathbb{N}$, then we call
$$(\ff,\gg)_{\lcm(\bar\ff,\bar\gg)}=
\Rnormed{\aa_1\wdots\aa_{p},\Lnormed{\ff,\cc_1\wdots\cc_{q}
}}
-\Rnormed{\bb_1\wdots\bb_{m},\Lnormed{\gg,\dd_1\wdots\dd_{n}
}}
$$
 an \emph{intersection composition}.
 \end{defi}

  A composition~$h$ in \ITEM1 of the form~$(zf)$ or~$(fz)$ is called \emph{trivial}  (modulo $\SS$) if~$h$ can be written as a linear combination of normal~$\SS$-polynomials such that their leading monomials are $\leq\ov\hh$. And a composition~$h$ in \ITEM2-\ITEM4 of the form~$(\ff_1,\ff_2)_{\mu}$ is called \emph{trivial} (modulo~$\SS$ with respect to~$\mu$) if~$\hh$ can be written as a linear combination of normal~$\SS$-polynomials such that their leading monomials are $<\mu$.
 Finally, for every~$\hh\in\BC(\XXX)$ and~$\mu\in \NF(\XXX)$,  the polynomial~$h$ is said to be \emph{trivial modulo $\SS$ with respect to~$\mu$}, denoted by
$$\hh \equiv0  \ \mod(\SS,\mu), $$  if~$\hh$ can be written as a linear combination of normal~$\SS$-polynomials such that their leading monomials are $<\mu$;

Now we show that, under certain conditions,  we can rewrite some~$S$-polynomials into normal $S$-polynomials.

\begin{lemm}\label{Ids}
Let $\SS$ be a monic subset of $\BC(\XXX)$ such that all the left multiplication compositions,  right multiplication compositions and multiplication compositions of $\SS$ are trivial. Then the followings hold:

\ITEM1 Every $\SS$-polynomial~$\hh$ can be written as a linear combination of normal~$S$-polynomials whose leading monomials are~$\leq \ov\hh$.

\ITEM2
If $\hh=\Rnormed{\zz_1\wdots\zz_{p-1},\zz_p,
\Lnormed{\ss,\yy_1,\yy_2\wdots\yy_q}}$, where $\ss\in\SS$,  $\ell(\overline{\ss})\geq2$, $\zz_1\wdots\zz_p,\yy_1\wdots\yy_q\in\XXX$
and~$\pp,\qq\in \mathbb{\NNN}^+$, then we have~$
\hh-\Rnormed{\zz_1\wdots\zz_{p-1},\Lnormed{\zz_p,\ss,\yy_1,\yy_2\wdots\yy_q}}
\equiv0  \ \mod(\SS,\ov\hh)$.
\end{lemm}

\begin{proof}\ITEM1 Since every element in~$S$ is a normal~$S$-polynomial, it suffices to show that, for every normal~$S$-polynomial~$\hh$, for every~$\mu\in \NF(\XXX)$,  the polynomial~$(\mu\hh)$ (resp. $(\hh\mu)$)
can be written as a linear combination of normal~$\SS$-polynomials whose leading monomials are less than or equal to $\overline{(\mu\hh)}$ (resp. $\overline{(\hh\mu)}$).

Assume that $\hh$ is a normal $\SS$-polynomial.
We use induction on $\ell(\mu)$ to prove \ITEM1.
Suppose that $\ell(\mu)=1$ and $\mu=\xx\in\XXX$.
If $\ell(\bar\hh)=1$, that is, $\hh=\ss\in\SS$ and $\ell(\bar\ss)=1$, then we have $(\mu\hh)=(\xx\ss)$ and $(\hh\mu)=(\ss\xx)$. Since all left multiplication compositions and  right multiplication compositions of $\SS$ are trivial, there is nothing to prove.
If  $\ell(\bar\hh)>1$, then we may assume
$$\hh=\Rnormed{\zz_1\wdots\zz_{p},\Lnormed{\ss,\yy_1\wdots\yy_q}},$$ where~$\zz_1\wdots\zz_{p},\yy_1\wdots\yy_q\in\XXX$ and  $\pp,\qq\in \mathbb{\NNN}$.
Thus~$(\mu\hh)=(\xx\hh)=\Rnormed{\xx,\zz_1\wdots\zz_{p},
\Lnormed{\ss,\yy_1\wdots\yy_q}}$ is already a normal $\SS$-polynomial.
Moreover, since~$<$ is a monomial order, it follows immediately that, for an arbitrary normal~$S$-polynomial~$f$,
for all~$\xx_{i_1}\wdots\xx_{i_t}\in \XXX$,  the polynomial~$\Rnormed{\xx_{i_1}\wdots\xx_{i_t},f}$ can be written as a linear combination of normal~$\SS$-polynomials whose leading monomials are less than or equal to $\ov{\Rnormed{\xx_{i_1}\wdots\xx_{i_t},\ff}}$.

 If $\qq=\pp=0$ or~$\qq>0$, then by Lemma~\ref{lemm-ass}\ITEM2, we have
 $$(\hh\mu)=(\ss\xx)  \mbox{  or  }  (\hh\mu)=(\hh\xx)=\Rnormed{\zz_1\wdots\zz_{p},
\Lnormed{\ss,\yy_1\wdots\yy_q,\xx}},$$ both of which are  already normal $\SS$-polynomials.  Note that every multiplication composition of $\SS$ is trivial. If~$\qq=0$ and~$\pp>0$, then we may assume
$$((\zz_{p}\ss)\xx)-(\zz_{p}(\ss\xx))=\sum_i \alpha_i \hh_i,$$
where each~$\alpha_i$ belongs to~$\kk$ and each~$\hh_i$ is a normal $\SS$-polynomial with~$\ov{\hh_i}<\ov{((\zz_{p}\ss)\xx)}$.
By  Lemma~\ref{lemm-ass}\ITEM2 again, we obtain
\begin{align*}
(\hh\mu)
&=(\Rnormed{\zz_1\wdots\zz_{p-1},\Lnormed{\zz_{p},\ss}}\xx)\\
&=\Rnormed{\zz_1\wdots\zz_{p-1},\Lnormed{\zz_{p},\ss,\xx}}\\
&=\Rnormed{\zz_1\wdots\zz_{p-1},((\zz_{p}\ss)\xx)}\\
&=\Rnormed{\zz_1\wdots\zz_{p-1},
(\zz_{p}(\ss\xx))
}+\sum_i \alpha_i\Rnormed{\zz_1\wdots\zz_{p-1},
\hh_i
}\\
&=\Rnormed{\zz_1\wdots\zz_{p},\Lnormed{\ss,\xx}}+\sum \gamma_t\hh''_t,
\end{align*}
where~$\Rnormed{\zz_1\wdots\zz_{p},\Lnormed{\ss,\xx}}$ is a normal~$\SS$-polynomial satisfying~$\ov{\Rnormed{\zz_1\wdots\zz_{p},\Lnormed{\ss,\xx}}}=\overline{(\hh\mu)}$,
and each $\gamma_t$ lies in~$\kk$, each $\hh''_t$ is a normal $\SS$-polynomial satisfying
$$\ov{\hh''_t} \leq\overline{\Rnormed{\zz_1\wdots\zz_{p-1},\hh_i}}
<\overline{\Rnormed{\zz_1\wdots\zz_{p-1},((\zz_{p}\ss)\zz)}}
=\overline{(\hh\xx)}=\overline{(\hh\mu)}.$$

Now suppose that $\ell(\mu)>1$ and $\mu=(\mu_1\mu_2)$, where $\mu_1,\mu_2\in\NF(\XXX)$. Since~$\ell(\mu_1)<\ell(\mu)$, by induction hypothesis, we may assume $$(\mu\hh)=((\mu_1\mu_2)\hh)=((\mu_1\hh)\mu_2)=\sum_i \alpha_i(\hh_i\mu_2),$$
where each $\alpha_i$ belongs to $\kk$, each $\hh_i$ is a normal $\SS$-polynomial satisfying~$\ov{\hh_i}\leq\overline{(\mu_1\hh)}$. Similarly we have~$(\hh_i\mu_2)=\sum_j \beta_{ij}\hh'_{ij}$, where each $\beta_{ij}\in\kk$, each~$\hh'_{ij}$ is a normal $\SS$-polynomial with~$\ov{\hh'_{ij}}\leq \overline{(\hh_i\mu_2)}\leq \overline{(\ov{(\mu_1\hh)}\mu_2)}=\overline{((\mu_1\hh)\mu_2)}
=\overline{(\mu\hh)}$.

Finally, note that~$(\hh\mu)=(\hh(\mu_1\mu_2))=(\mu_1(\hh\mu_2))$, the claim follows immediately.

\ITEM2
We use induction on $\pp+\qq$.
If $\pp=\qq=1$, then since every multiplication composition of $\SS$ is trivial,
 we have
$$\Rnormed{\zz_1,\Lnormed{\ss,\yy_1}}-
\Lnormed{\zz_1,\ss,\yy_1}\equiv 0 \mod(\SS,\overline{\Rnormed{\zz_1,
\Lnormed{\ss,\yy_1}}}).$$

In particular, if $\pp+\qq>2$, then we assume
$$((\zz_p(\ss\yy_1))
-((\zz_p\ss)\yy_1)) =\sum_i\beta_ih_i',$$
where each~$\beta_i$ lies in~$k$ and each~$h_i'$ is a normal~$S$-polynomial satisfying~$\ov{h_i'}\leq \ov{((\zz_p(\ss\yy_1))}$.
By Lemma~\ref{lemm-ass}, we have
\begin{align*}
&\hh
-\Rnormed{\zz_1\wdots\zz_{p-1},\Lnormed{((\zz_p\ss)\yy_1),\yy_2\wdots\yy_q}}\\
=&  \Rnormed{\zz_1\wdots\zz_{p-1},\Lnormed{(\zz_p(\ss\yy_1)),\yy_2\wdots\yy_q}}
-\Rnormed{\zz_1\wdots\zz_{p-1},\Lnormed{((\zz_p\ss)\yy_1),\yy_2\wdots\yy_q}}\\
=&\Rnormed{\zz_1\wdots\zz_{p-1},\Lnormed{((\zz_p(\ss\yy_1))
-((\zz_p\ss)\yy_1)),\yy_2\wdots\yy_q}}\\
=&\sum_i \beta_i\Rnormed{\zz_1\wdots\zz_{p-1},\Lnormed{\hh'_i,\yy_2\wdots\yy_q}}.
\end{align*}
Since~$<$ is a monomial order,  the claim follows by~\ITEM1.
\end{proof}

The following lemma says that, under certain conditions, two normal~$S$-polynomials with the same leading monomial are not of significant difference.

\begin{lemm}\label{key-lemma}
Let $\SS$ be a monic subset of $\BC(\XXX)$ such that all compositions of $S$ are trivial, and let~$\hh_{1}$, $\hh_{2}$ be two normal $\SS$-polynomials. If~$\overline{\hh_{1}}=\overline{\hh_{2}}$, then we have
$$\hh_{1}-\hh_{2}\equiv0 \mod (\SS,\overline{h_{1}}).$$
\end{lemm}

\begin{proof}
If~$\ell(\overline{\hh_1})=1=\ell(\overline{\hh_2})$, then we may assume~$\hh_{1}=\ss_1$ and~$\hh_{2}=\ss_2$, where~$\ss_1,\ss_2\in\SS$.  Since every inclusion composition of $\SS$ is trivial, we have~$\hh_{1}-\hh_{2}\equiv0 \mod (\SS,\overline{h_{1}})$.
Now we suppose $\ell(\overline{\hh_1})>1$. Then we may assume
$$\hh_1=\Rnormed{\zz_1\wdots\zz_p,
\Lnormed{\Rnormed{\aa_1\wdots\aa_{m},
\Lnormed{\ss_1,\bb_1\wdots\bb_n
}
},\yy_1\wdots\yy_q
}
},$$
$$\hh_2=\Rnormed{\zz_1\wdots\zz_p,
\Lnormed{\Rnormed{\aa'_1\wdots\aa'_{\mm'},
\Lnormed{\ss_2,\bb'_1\wdots\bb'_{\nn'}
}
},\yy_1\wdots\yy_q
}
},$$
where $\ss_1,\ss_2\in\SS$, $\ell(\overline{\ss_1}), \ell(\overline{\ss_2})>1$, $\zz_1\wdots\zz_p,\yy_1\wdots\yy_q,
\aa_1\wdots\aa_{m},\bb_1\wdots\bb_n,
\aa'_1\wdots\aa'_{\mm'},\bb'_1\wdots\bb'_{\nn'}$ lie in~$\XXX$, $\pp,\qq,\nn,\mm,\nn',\mm'\in \mathbb{N}$ and
$\{\aa_1\wdots\aa_{m}\}\cap\{\aa'_1\wdots\aa'_{\mm'}\}=
\{\bb_1\wdots\bb_n\}\cap\{\bb'_1\wdots\bb'_{\nn'}\}=\emptyset$.
In particular, we deduce
$$\lcm(\overline{\ss_1},\overline{\ss_2})=\overline{\Rnormed{\aa_1\wdots\aa_{m},
\Lnormed{\ss_1,\bb_1\wdots\bb_n}}}=
\overline{\Rnormed{\aa'_1\wdots\aa'_{\mm'},
\Lnormed{\ss_2,\bb'_1\wdots\bb'_{\nn'}}}}.$$
There are two cases to consider.
If $\nn+\mm<\ell(\overline{\ss_2})$, then
we have
$$(\ss_1,\ss_2)_{\lcm(\overline{\ss_1},\overline{\ss_2})}=
 \Rnormed{\aa_1\wdots\aa_{m},\Lnormed{\ss_1,\bb_1\wdots\bb_n}}-
 \Rnormed{\aa'_1\wdots\aa'_{\mm'},\Lnormed{\ss_2,\bb'_1\wdots\bb'_{\nn'}}}
 =\sum_j \beta_j\hh'_j,$$ where
each~$\beta_j$ lies in~$\kk$, each $\hh'_j$ is a normal $\SS$-polynomial satisfying~$\ov{\hh'_j}<\lcm(\overline{\ss_1},\overline{\ss_2})$.
So we obtain
$$
\hh_{1}-\hh_{2}
=\Rnormed{\zz_1\wdots\zz_p,
\Lnormed{(\ss_1,\ss_2)_{\lcm(\overline{\ss_1},
\overline{\ss_2})},\yy_1\wdots\yy_q}}
=\sum_j \beta_j\Rnormed{\zz_1\wdots\zz_p,
\Lnormed{\hh'_j,\yy_1\wdots\yy_q}}.
$$
By Lemma~\ref{Ids}\ITEM1, each $\Rnormed{\zz_1\wdots\zz_p,
\Lnormed{\hh'_j,\yy_1\wdots\yy_q}}$ can be written as a linear combination of normal $\SS$-polynomials such that their leading monomials are
$$\leq\overline{
\Rnormed{\zz_1\wdots\zz_{p-1},\zz_p,
\Lnormed{\hh'_j,\yy_1\wdots\yy_q}}}<
\overline{\Rnormed{\zz_1\wdots\zz_p,
\Lnormed{\lcm(\overline{\ss_1},\overline{\ss_2}),\yy_1\wdots\yy_q}}}
=\overline{h_{1}}.$$
Therefore, we have $\hh_{1}-\hh_{2}\equiv0 \mod (\SS,\overline{h_{1}})$.

If $n + m =\ell(\overline{s_2})$, then we deduce~$n,m,n',m'\in \mathbb{N}^{+}$,
$\overline{\ss_2}=\Rnormed{\aa_1\wdots\aa_{m-1},
\Lnormed{\aa_{m},\bb_1\wdots\bb_n
}
}$ and $\overline{\ss_1}=\Rnormed{\aa'_1\wdots\aa'_{\mm'-1},
\Lnormed{\aa'_{\mm'},\bb'_1\wdots\bb'_{\nn'}
}
}$.
Thus
we have
$$(\ss_1\overline{\ss_2})
=\Rnormed{\aa_1\wdots\aa_{m},\Lnormed{\ss_1,\bb_1\wdots\bb_n}},\
(\overline{\ss_1}\ss_2)=\Rnormed{\aa'_1\wdots\aa'_{\mm'},
\Lnormed{\ss_2,\bb'_1\wdots\bb'_{\nn'}
}}$$
and
$$\hh_{1}=\Rnormed{\zz_1\wdots\zz_{p-1},\zz_p,
\Lnormed{\ss_1,\overline{\ss_2},\yy_1\wdots\yy_q}}.$$
By Lemma \ref{Ids}\ITEM2,
we have
\begin{align*}
\hh_{2}&=\Rnormed{\zz_1\wdots\zz_p,\aa'_1\wdots\aa'_{\mm'},
\Lnormed{\ss_2,\bb'_1\wdots\bb'_{\nn'},\yy_1\wdots\yy_q
}}\\
&=\Rnormed{\zz_1\wdots\zz_p,\aa'_1\wdots\aa'_{\mm'-1},
\Lnormed{\aa'_{\mm'},\ss_2,\bb'_1\wdots\bb'_{\nn'},\yy_1\wdots\yy_q
}}
+\sum_i \alpha_i\hh'_i\\
&=\Rnormed{\zz_1\wdots\zz_{p-1},\zz_p,
\Lnormed{\overline{\ss_1},\ss_2,\yy_1\wdots\yy_q
}}+\sum_i \alpha_i\hh'_i,
\end{align*}
 where each $\alpha_i$ lies in~$\kk$,  and each~$\hh'_i$ is a normal $\SS$-polynomial such that $\overline{\hh'_i}<\overline{\hh_{2}}$.
Finally we obtain
\begin{align*} &\hh_{1}-\hh_{2}&\\
=&\Rnormed{\zz_1\wdots\zz_{p-1},\zz_p,
\Lnormed{\ss_1,\overline{\ss_2},\yy_1\wdots\yy_q}}-
\Rnormed{\zz_1\wdots\zz_{p-1},\zz_p,
\Lnormed{\overline{\ss_1},\ss_2,\yy_1\wdots\yy_q
}}-\sum_i \alpha_i\hh'_i&\\
=&\Rnormed{\zz_1\wdots\zz_{p-1},\zz_p,
\Lnormed{\ss_1,\overline{\ss_2},\yy_1\wdots\yy_q}}+
\Rnormed{\zz_1\wdots\zz_{p-1},\zz_p,
\Lnormed{\ss_1,\ss_2,\yy_1\wdots\yy_q}}&\\
&-
\Rnormed{\zz_1\wdots\zz_{p-1},\zz_p,
\Lnormed{\ss_1,\ss_2,\yy_1\wdots\yy_q}}-
\Rnormed{\zz_1\wdots\zz_{p-1},\zz_p,
\Lnormed{\overline{\ss_1},\ss_2,\yy_1\wdots\yy_q
}}-\sum_i \alpha_i\hh'_i&\\
=&\Rnormed{\zz_1\wdots\zz_{p-1},\zz_p,
\Lnormed{\ss_1,(\overline{\ss_2}-\ss_2), \yy_1\wdots\yy_q}}&\\
&+\Rnormed{\zz_1\wdots\zz_{p-1},\zz_p,
\Lnormed{(\ss_1-\overline{\ss_1}),\ss_2,\yy_1\wdots\yy_q}}
-\sum_i \alpha_i\hh'_i .&
\end{align*}
Since we have
 $$\ov{\Rnormed{\zz_1\wdots\zz_{p-1},\zz_p,
\Lnormed{\ss_1,(\overline{\ss_2}-\ss_2), \yy_1\wdots\yy_q}}}
<\ov{\Rnormed{\zz_1\wdots\zz_{p-1},\zz_p,
\Lnormed{\ss_1,\overline{\ss_2}, \yy_1\wdots\yy_q}}}
=\ov{h_1}$$
and $$\ov{\Rnormed{\zz_1\wdots\zz_{p-1},\zz_p,
\Lnormed{(\ss_1-\overline{\ss_1}),\ss_2,\yy_1\wdots\yy_q}}}
<\ov{\Rnormed{\zz_1\wdots\zz_{p-1},\zz_p,
\Lnormed{ \overline{\ss_1},\ss_2,\yy_1\wdots\yy_q}}}
=\ov{h_1},$$
the lemma follows by Lemma~\ref{Ids}\ITEM1.
\end{proof}

 Now we are ready to prove the Composition-Diamond lemma for bicommutative algebras.
\begin{theorem}\label{cd-lemma for BC}  \emph{(}Composition-Diamond lemma for bicommutative algebras\emph{)}
Let $\SS$ be a monic subset of $\BC(\XXX)$. Let $\Id\SS$ be the ideal of $\BC(\XXX)$ generated by $\SS$. Then the followings are equivalent.

\ITEM1 All the compositions of~$\SS$ are trivial.

\ITEM2 The set~$\SS$ is a Gr\"{o}bner--Shirshov basis  for $\BC(\XXX)/\Id\SS$, that is, for every nonzero polynomial~$\ff\in \Id\SS$,  we have~$\overline {\ff}=\overline{\hh}$ for some normal~$\SS$-polynomial~$\hh$.

\ITEM3 The set $\Irr(\SS)=\NF(\XXX)\sm\{\mu\in \NF(\XXX)\mid \mu= \overline{\hh}, \hh \mbox{ is a normal }\SS\mbox{-polynomial} \}$ forms  a linear basis of the algebra $\BC(\XXX | \SS):=\BC(\XXX)/\Id\SS$.

\end{theorem}

\begin{proof}
 \ITEM1 $\Rightarrow$ \ITEM2     For every  nonzero element~$\ff$ in~$\Id\SS$,  by Lemma~\ref{Ids}\ITEM1, we may assume that~$\ff=\sum_{i=1}^n\alpha_i\hh_{i}$,
  where each
$\alpha_i\in \kk, \alpha_i\neq 0$ and each $\hh_{i}$ is a normal $\SS$-polynomial.
Define~$
\mu_i= \overline{\hh_{i}}$ and assume that~$\mu_1=\mu_2= \cdots=\mu_{_l}>\mu_{l+1}\geq \cdots.
$
 Use induction on~$\mu_1$.
 If~$\mu_1=\ov{\ff}$, then there is nothing to prove.
 If~$\mu_1>\ov{\ff}$, then we have~$\sum_{i=1}^l\alpha_{i}=0$ and
$$
\ff
  =\sum_{ i=1}^l\alpha_i\hh_{1}
  -\sum_{ i=1}^l\alpha_i(\hh_{1}
  -   \hh_{i})+\sum_{ i=l+1}^n\alpha_i\hh_{i}
  =\sum_{j }\beta_{j}\hh_{j}'+\sum_{ i=l+1}^n\alpha_i\hh_{i},
$$
 where each $\ov{\hh_{j}'}<  \mu_1$ by Lemma \ref{key-lemma}.
 Claim \ITEM2 follows by induction hypothesis.

\ITEM2 $\Rightarrow$ \ITEM3  By Lemma \ref{f=Irr+n-s-polynomials},
$\Irr(\SS)$ generates
$\BC(\XXX|\SS)$ as a vector space. Suppose that~$\sum_{i}\alpha_i\mu_i=0$ in $\BC(\XXX|\SS)$,
where each $ \alpha_i\in \kk\sm \{0\}$ and each~$\mu_i\in  \Irr(\SS)$,  $\mu_1>\mu_2> \pdots$. Then we have~$0\neq \sum_{i}\alpha_i\mu_i\in \Id\SS$. Thus~$\overline{\sum_{i}\alpha_i\mu_i}=\mu_1\in \Irr(\SS)$,
 which contradicts with~\ITEM2.

\ITEM3 $ \Rightarrow$ \ITEM1  Every composition of~$\SS$
   is an element in $\Id\SS$. By Lemma~\ref{f=Irr+n-s-polynomials} and \ITEM3,
     we obtain  \ITEM1.
 \end{proof}

\subsection{Finite Gr\"obner-Shirshov bases}
Our aim in this subsection is to show that every finitely generated bicommutative algebra has a finite \gsb.

\begin{defi}
Let $\SS$ be a Gr\"{o}bner--Shirshov basis for $\BC(\XXX|\SS)$.  Then we call $\SS$ a \emph{minimal Gr\"{o}bner--Shirshov basis}, if for every $\ss\in \SS$,  we have~$\ov\ss\in \Irr(S\setminus \{s\})$.
\end{defi}

 The following lemma is a general fact for various kind of algebras, but we still offer a proof for the convenience of the readers.
\begin{lemm}\label{lIrr}
Let $I$ be an ideal of $\BC(\XXX)$ and $\SS$ a Gr\"{o}bner--Shirshov basis for $\BC(\XXX)/I$. For every~$T\subseteq S$,
if $\Irr(\TT)=\Irr(\SS)$ then $\TT$ is also a Gr\"{o}bner--Shirshov basis for $\BC(\XXX)/I$.
\end{lemm}
\begin{proof} For every~$\ff\in I$, since $\Irr(\TT)=\Irr(\SS)$ and $\SS$ is a Gr\"{o}bner--Shirshov basis for~$\BC(\XXX)/I$, we have $ \overline{\ff}=\overline{\hh}=\overline{\hh'}$ for some normal $S$-polynomial~$\hh$ and for some normal~$T$-polynomial~$\hh'$.  So we obtain~$\ff_1=\ff-\mathsf{lc}(\ff)\hh'\in I$ and $\overline{\ff_1}<\overline{\ff}$. By induction on~$\overline{\ff}$, we deduce that~$\ff$ can be written as a linear combination of normal $\TT$-polynomials, in particular,  we have $\ff\in \Id\TT$. This shows that $I= \Id\TT$. So the result follows from Theorem \ref{cd-lemma for BC}.
\end{proof}

 Let~$I$ be an ideal of~$\BC(\XXX)$.  Then the set~$\{f\in I \mid \mathsf{lc}(\ff)=1\}$ forms a Gr\"obner-Shirshov basis for~$\BC(\XXX)/I$, so Gr\"obner-Shirshov bases always exist.      Now we show that minimal Gr\"{o}bner--Shirshov bases always exist.
\begin{lemm}\label{lemma-minimal}
Let~$S$ be a Gr\"{o}bner--Shirshov basis for~$\BC(\XXX)/\Id{S}$. Then there exists a set~$T\subseteq S$ such that~$T$ is a minimal Gr\"{o}bner--Shirshov basis for the algebra~$\BC(\XXX)/\Id{S}$.
\end{lemm}

\begin{proof}
Without loss of generality, assume that~$S$ is a nonempty set.
For each~$\mu\in \overline{\SS}:=\{\overline{s}\in \NF(\XXX) \mid s\in S\}$, we fix an arbitrary polynomial $\ff_{\mu}$ in $\SS$ such that $\overline{\ff_{\mu}}=\mu$. Define
$$
\SS'=\{\ff_{\mu}\in \SS\mid \mu\in \overline{\SS}\}.
$$
Then it is clear that $ \SS\supseteq \SS'$ and~$\Irr(\SS')=\Irr(\SS)$.
By Lemma \ref{lIrr}, $\SS'$ is a Gr\"{o}bner--Shirshov basis for~$\BC(\XXX)/\Id{S}$.
Define
$$\SS_1=\{\ss\in\SS'\mid \ell(\bar\ss)=1\},\ \SS_2=\{\ss\in\SS'\mid \ell(\bar\ss)>1\}.$$
Then we have $\SS'=\SS_1\cup\SS_2$.

Let $\mu_0=\minn\overline{\SS_2}$ be the minimal element in~$\overline{\SS_2}$. Define
$\SS_{\mu_0} =\{\ss \in \SS_2\mid \bar\ss=\mu_0\}$ and define~$\SS_{\mu} =\emptyset$ for every~$\mu<\mu_0$ in~$\NF(\XXX)\sm\XXX$.
Suppose that for every~$\nu<\mu$  in~$\NF(\XXX)\sm\XXX$, the set~$\SS_{\nu}$ has already been defined.   Then we define
\begin{displaymath}
\SS_{\mu}=
\begin{cases}
\cup_{\nu<\mu}\SS_{\nu}
 \ \ \ \ \  \ \ \ \ \ \ \ \ \text{ if $\mu\not\in \Irr(\cup_{\nu<\mu}\SS_{\nu})\cap \overline{S}$,} \\
\cup_{\nu<\mu}\SS_{\nu}\cup \{\ff_{\mu}\}  \ \ \ \  \text{ if $\mu\in \Irr(\cup_{\nu<\mu}\SS_{\nu})\cap \overline{S}$}.
\end{cases}
\end{displaymath}
Define
$$
\SS''=\SS_1\cup(\bigcup_{\mu \in \NF(\XXX)} \SS_{\mu}).
$$

We claim that~$\SS''$ is a minimal Gr\"obner-Shirshov basis for~$\BC(\XXX)/\Id\SS$.
By Lemma~\ref{lIrr}, it suffices to show $\Irr(\SS'')= \Irr(\SS')$. Since $\SS''\subseteq \SS'$,
 it suffices to show~$\Irr(\SS'')\subseteq \Irr(\SS')$, or equivalently, $\NF(\XXX)\sm \Irr(\SS')\subseteq \NF(\XXX)\sm \Irr(\SS'')$. Assume that~$\mu=\ov{\hh'}\in \NF(\XXX)\sm \Irr(\SS')$ and assume that~$\hh'$ is a normal $\SS'$-polynomial such that
  $$\hh'=\Rnormed{\zz_1\wdots\zz_p,
\Lnormed{\ss',\yy_1\wdots\yy_q
}
},$$
where~$\ss'\in\SS'$, $\zz_1\wdots\zz_p,\yy_1\wdots\yy_q\in\XXX$ and $p,q\in\mathbb{N}$.
 If~$\ss'\in\SS''$, then~$\mu\in\NF(\XXX)\sm \Irr(\SS'')$.
If~$\ss'\in S'\setminus S''$, then by the construction of~$S''$ we deduce~$\ell(\ov{\ss'})>1$ and $\overline{\ss'}=\ov{\hh''}$ for some  normal~$\SS''$-polynomial $\hh''$. So we may assume
 $$
\hh''=\Rnormed{\aa_1\wdots\aa_n,
\Lnormed{\ss'',\bb_1\wdots\bb_m
}
},$$
where $\aa_1\wdots\aa_n,\bb_1\wdots\bb_m\in\XXX$ and $m,n\in\mathbb{N}, m+n\in\mathbb{N}^+$, $\ss''\in S''$ and~$\ell(\overline{s''})\geq 2$.
Set
$$\ff=\Rnormed{\zz_1\wdots\zz_p,
\Lnormed{\hh'',\yy_1\wdots\yy_q
}
}.$$

 If $\mm\in\mathbb{N}^+$, then $\ff=
\Rnormed{\zz_1\wdots\zz_p,\aa_1\wdots\aa_n,
\Lnormed{\ss'',\bb_1\wdots\bb_m,\yy_1\wdots\yy_q
}
}$ is a normal $\SS''$-polynomial.
Therefore we obtain $\mu=\bar\ff\in\NF(\XXX)\sm \Irr(\SS'')$.
If $\mm=0$, then  since $\ss''\in\SS''\subseteq\SS'$, $\SS'$ is a \gsb\ in $\BC(\XXX)$ and Lemma \ref{Ids}\ITEM2 holds,
 we have
\begin{align*}
  \ff&=
\Rnormed{\zz_1\wdots\zz_p,\aa_1\wdots\aa_{n-1},
\Lnormed{\aa_n,\ss'',\yy_1\wdots\yy_q
}
}&\\
&=\Rnormed{\zz_1\wdots\zz_p,\aa_1\wdots\aa_n,
\Lnormed{\ss'',\yy_1\wdots\yy_q
}
}
+\sum_j \alpha_j\hh_j,&
\end{align*}
where each $\alpha_j\in\kk$, each $\hh_j$ is a normal $\SS'$-polynomial and~$\ov{\hh_j}<\bar\ff$.
Therefore we have  $$\mu=\bar\ff=\overline{\Rnormed{\zz_1\wdots\zz_p,\aa_1\wdots\aa_n,
\Lnormed{\ss'',\yy_1\wdots\yy_q
}
}}
\in\NF(\XXX)\sm \Irr(\SS'').$$

By Lemma \ref{lIrr}, $\SS''$ is a  Gr\"{o}bner--Shirshov basis for $\BC(\XXX)/\Id{S}$.
Moreover, by the construction of~$S''$, we obtain that~$\SS''$ is a minimal Gr\"{o}bner--Shirshov basis for $\BC(\XXX)/\Id{S}$.
\end{proof}

As a corollary of the Hilbert's Theorem, see for example~\cite[Sect.1.6]{gsbook1}, every finitely generated commutative algebra has a finite Gr\"obner  basis.

\begin{lemm}\label{divide}
Let $U$ be an infinite subset of $[\XXX]^+$, where $\XXX$ is a finite set. Then there exists  an infinite subset~$\{u_i\in U\mid i\in \mathbb{N}^+\}$ of $U$
such that $u_{i+1}$ is a multiple of $u_i$ for every~$i\in \mathbb{N}^+$.
\end{lemm}
\begin{proof}
Let~$\Id{U}$ be the ideal of~$k[X]$ generated by~$U$ and let~$k[X|U]$ be the quotient algebra~$k[X]/\Id{U}$.
By Hilbert's Theorem, there exists a finite set~$U_1\subset U$ that generates~$\Id{U}$, namely, we have~$\Id{U_1}=\Id{U}$.
Moreover, it is easy to see that~$U_1$ is a \gb\ for~$k[X]/\Id{U}$.
Since~$\UUU$ is infinite, there exists a monomial~$u_1\in U_1$ such that infinite many multiples of~$u_1$ lies in~$U$.
Define~$V_2=\{u\in U\mid u_1| u, u\neq u_1\}$ and let~$U_2$ be a finite subset of~$V_2$ such that~$V_2$ is a Gr\"obner basis for~$k[X|V_2]$.
Then we can choose an element~$u_2\in U_2$ such that there are infinitely many multiples of~$u_2$ in~$V_2$.
Continue this process, we shall obtain the desired infinite set.
\end{proof}

\begin{theorem}\label{finitegsb}
Let $\XXX$ be a finite set and let $I$
 be an ideal of $\BC(\XXX)$.
Then there exists a  finite Gr\"{o}bner--Shirshov basis for~$\BC(\XXX)/I$.
\end{theorem}
\begin{proof}
  By Lemma~\ref{lemma-minimal}, we may assume that~$\SS$ is a minimal \gsb\ for $\BC(\XXX)/I$.   We shall prove that~$S$ is finite. Suppose for the contrary  that~$S$ is an infinite set.

 Define~$\SS_1=\{\ss\in\SS\mid \ell(\bar\ss)=1\}, \ {\SS_2=\{\ss\in\SS\mid \ell(\bar\ss)>1\}}$.
 Then  we have
 $\SS=\SS_1\cup\SS_2$.
 Since $\SS$ is a minimal \gsb\ for $\BC(\XXX)/I$ and~$X$ is a finite set, we obtain that~$\SS_1$ is a finite set while~$\SS_2$ is infinite.
 Define
 $$\LLL=\{\uu\in[\XXX]^+\mid \mbox{ there exists } \vv\in[\XXX]^+\mbox{ such that }[\uu;\vv]\in\overline{\SS_2}\}$$
 and
 $$\RRR=\{\vv\in[\XXX]^+\mid \mbox{ there exists } \uu\in[\XXX]^+\mbox{ such that }[\uu;\vv]\in\overline{\SS_2}\}.$$

 Without loss of generality, we assume that $\LLL$ is an infinite set. By Corollary \ref{divide}, there exists an infinite set
 $$\LLL_1=\{u_i\in L\mid i\in \mathbb{N}^+\}$$
such that $u_{i+1}$ is a multiple of $u_i$ for every~$i\in \mathbb{N}^+$.
Define $$\SS'=\{[\uu;\vv]\mid \uu\in \LLL_1, \vv\in[\XXX]^+, [\uu;\vv]\in\overline{\SS}\}.$$
Obviously, the set $\SS'$ is infinite.
Define
$$\RRR_1=\{\vv\in[\XXX]^+\mid \mbox{ there exists } \uu\in[\XXX]^+\mbox{ such that } [\uu;\vv]\in\SS'\}.$$

   If $\RRR_1$ is finite, then for some monomial~$\vv\in\RRR_1$, there exist some monmomials~$u_i,u_j\in L_1$  such that~$[\uu_i;\vv], [\uu_j;\vv]\in\SS'\subseteq\overline{\SS}$ and~$u_i|u_j$.
 Thus~$ [\uu_j;\vv]$ is a multiple of~$[\uu_i;\vv]$, which contradicts to the assumption that~$\SS$ is a minimal \gsb\ for $\BC(\XXX)/I$.

If~$\RRR_1$ is infinite, then by Corollary \ref{divide}  there exists an infinite  set $\RRR_2=\{\vv_i\in R_1\mid i\in \mathbb{N}^+\}$ such that $\vv_{i+1}$ is a multiple of $\vv_i$ in $[\XXX]^+$  for
 every~$i\in \mathbb{N}^+$.
 We claim that for every monomial~$u_i$ in~$L_1$, there exists at most one monomial~$v_{j_i}\in R_2$ such that~$[u_i; v_{j_i}]\in S'$. For otherwise, say~$[u_i;v_p], [u_i;v_q]\in S'\subseteq \bar\SS$ and~$p<q$, then we have~$[u_i;v_p] | [u_i;v_q]$, which contradicts to the assumption that~$S$ is a minimal \gsb. Similarly, for every monomial~$v_i$ in~$R_2$, there exists exactly one monomial~$u_{j_i}\in L_1$ such that~$[u_{j_i};v_i]\in S'$. Assume~$[u_p;v_1]\in\SS'\subseteq\overline{\SS}$. Since~$\{u_1\wdots u_p\}$ is a finite set, by the above reasoning, for some large enough integer~$\mm>1$, there exists some element~$[u_n;v_m]\in S'\subseteq \bar\SS$  such that $n>p$.
 Therefore~$[\uu_{\nn};\vv_{\mm}]$ is a multiple of $[\uu_{p};\vv_1]$, which  contradicts to the assumption that $\SS$ is a minimal \gsb\ for~$\BC(\XXX)/I$.
\end{proof}
As a corollary, we obtain an alternative proof of the following results.
\begin{coro}\cite{bicom-variety}
Finitely generated
bicommutative algebras are weakly noetherian, i.e., satisfy the ascending chain
condition for two-sided ideals.
\end{coro}
\begin{coro} The word problem of an arbitrary finitely presented bicommutative algebra is solvable. \end{coro}

We conclude this section with a few comments on the Gr\"obner-Shirshov bases theory we established for bicommutative algebras. In Definition~\ref{defi-gsb}\ITEM1, when~$\bar{f}$ is a letter, certainly we can define~$\Rnormed{\zz_1\wdots\zz_{m},\Lnormed{\ff,\yy_1\wdots\yy_{n}
}}$ and $\Rnormed{\zz_1\wdots\zz_{m-1},\Lnormed{\zz_{m},\ff,\yy_1\wdots\yy_{n}
}}$ to be normal~$S$-polynomials, where~$z_1\wdots y_n\in X$. But then many of the proofs in the article will become much more lengthy with boring discussions due to this alternative way of defining normal~$S$-polynomials.  And in our way, there is also some disadvantage. For example, let~$S=\{f\}$ and~$\bar\ff=x\in X$. Then one naturally expects~$S$ to be a \gsb. However, in our way, this is not the case. We have to add all the relations of the form~$[f;y]$ and $[y;f]$ for every~$y\in X$ to~$S$ in order to form a \gsb. We accept this defect because adding these polynomials just means that every time we can rewrite~$\bar\ff$ into $\bar\ff-f$ in its appearance in $\Rnormed{\zz_1\wdots\zz_{m},\Lnormed{\ff,\yy_1\wdots\yy_{n}
}}$ and $\Rnormed{\zz_1\wdots\zz_{m-1},\Lnormed{\zz_{m},\ff,\yy_1\wdots\yy_{n}
}}$, which is in fact the same as the ordinary way of dealing with \gsb.

\section{\GK\ of bicommutative algebras}\label{sec-main-result}
Our aim in this section is to prove the main result of this article, namely, the Gelfand-Kirillov dimension of an arbitrary finitely generated bicommutative algebra is a nonnegative integer. The method we applied here is the Gr\"obner-Shirshov bases theory for bicommutative algebras, which  is established in Section \ref{gsbforbi}.
Moreover, we show that, in order to calculate the Gelfand-Kirillov dimension of a finitely generated bicommutative algebra, it suffices to investigate some property  of a finite set instead of constructing a possibly infinite linear basis for the bicommutative algebra under consideration.

We begin with some notation. Let~$\VVV, \VVV_1$ and~$\VVV_2$ be vector subspaces of a bicommutative algebra~$\MB$. We define
$$(\VVV_1\VVV_2)=\Span_k\{(\xx\yy)\mid \xx\in \VVV_1, \yy\in \VVV_2\}.$$
Then we define~$\VVV^1=\VVV$ and~$\VVV^{n}=\sum_{1\leq i\leq n-1}(\VVV^i\VVV^{n-i})$ for every~$\nn\geq 2$. Finally, we define
$$
\VVV^{\leq n}=\VVV^{1}+\VVV^{2}+\pdots+\VVV^{n}.
$$

Let~$\MB$ be a bicommutative algebra. Then   the \GK\ of~$\MB$ is defined to be
$$ \gkd\MB=\sup\limits_{\VVV}\overline{\lim\limits_{\nn \to \infty}}\log_{\nn}\mathsf{dim}(\VVV^{\leq\nn}),$$
where the supremum is taken over all finite dimensional subspaces of~$\MB$.
 It follows immediately that, for a bicommutative algebra~$\MB=\BC(\XXX|\SS)$, where $\XXX$ is a finite set and $\SS$ is a \gsb\ for $\BC(\XXX|\SS)$, we have
$$\gkd\MB=\overline{{\lim\limits_{\nn\to \infty}}}\log_{\nn}
\#  \{\mu\in\Irr(\SS)\mid \ell(\mu)\leq\nn\}.$$

Let~$\XXX=\{\xx_1,\dots,\xx_\rr\}$ be a nonempty finite set and let~$\SS$ be a finite \gsb\ for~$\BC(\XXX|\SS)$.
Define
\begin{eqnarray*}
\NNN_{\SS}=
\begin{cases}
1,& \mbox{if } \SS=\emptyset\\
\minn\{\nn\mid \bar\ss | [\xx_1^n\cdots\xx_{\rr}^n;\xx_1^n\cdots\xx_{\rr}^n], \ss\in \SS\},& \mbox{if } \SS\neq\emptyset.
\end{cases}
\end{eqnarray*}
For every $\xx\in\XXX$, we define
$$\Sd(\xx)=0,$$
 and  for every monomial $\mu=[\xx_1^{n_1}\cdots\xx_{\rr}^{n_{\rr}};\xx_1^{m_1}\cdots\xx_{\rr}^{m_{\rr}}]
 \in\NF(\XXX)\sm \XXX$,
we define
$$
\Sd(\mu)=
\#\{i\mid \nn_{i}>\NNN_{\SS},i\in\{1\wdots r\}\}
+\#\{i\mid \mm_{i}>\NNN_{\SS},i\in\{1\wdots r\}\}.$$
Finally, for an arbitrary nonempty set~$B\subseteq \Irr(\SS)$, we define
$$\Sd(B)=\max\{\Sd(\mu)\mid \mu\in B\}$$
and define
$$\Sd(\emptyset)=0.$$
We shall soon see that~$\Sd(B)$ is always finite if~$X$ is a finite set.
Now we investigate the \GK\  of a finitely generated bicommutative algebra.
 \begin{lemm}\label{lemma-GKBC}
Let  $\XXX$ be a nonempty finite set and let~$\SS$ be a finite \gsb\ for $\BC(\XXX|S)$. Then we have
$\gkd{\BC(\XXX|S)}=\Sd(\Irr(\SS))$.
\end{lemm}
\begin{proof}
Assume $\XXX=\{\xx_1,\dots,\xx_{\rr}\}$.
Define~$\MMM=\Sd(\Irr(\SS))$.
If $\MMM=0$, then we have~$\Irr(\SS)\subseteq
X\cup \{[\xx_1^{n_1}\cdots\xx_{\rr}^{n_{\rr}};
\xx_1^{m_1}\cdots\xx_{\rr}^{m_{\rr}}]\in\NF(\XXX)\sm\XXX\mid n_i,m_j\leq N_S, 1\leq i,j\leq r\}$. It follows that $\gkd{\BC(\XXX|S)}=0=\Sd(\Irr(\SS)\sm\XXX)$.

Now we suppose~$\MMM\geq 1$. Since $\MMM=\Sd(\Irr(\SS))$, for every $\mu\in \Irr(\SS)\sm \XXX\subseteq\NF(\XXX)\sm \XXX$, we have $\Sd(\mu)\leq\MMM$.
Therefore, we have
$$\Irr(\SS)\subseteq\{\mu\in\NF(\XXX)\mid \Sd(\mu)\leq\MMM\}. $$
It follows immediately  that we have
\begin{equation}\label{b1}
\gkd{\BC(\XXX|\SS)}\leq M.
\end{equation}
Moreover, since~$\Sd(\Irr(S)\sm \XXX)=M$, there exists some monomial~$\mu\in\Irr(\SS)\sm \XXX$ such that~$\Sd(\mu)=\MMM$.
Without loss of generality, we assume $\mu=[\xx_1^{n_1}\cdots\xx_{\rr}^{n_{\rr}};\xx_1^{m_1}\cdots\xx_{\rr}^{m_{\rr}}]$, where~$\nn_{\ii_1},\dots,\nn_{\ii_p},\mm_{\jj_1},\dots,\mm_{\jj_q}>\NNN_{\SS}$,
 $p+q=\MMM$, $ p,q\in\mathbb{N}$,  $\nn_{l},\mm_{l'} \in \{0,1\}$ for all~$l\in \{1\wdots r\}\setminus\{i_1\wdots i_p\} $ and~$l'\in \{1\wdots r\}\setminus\{j_1\wdots j_q\}$.
There are three cases to consider.

  If $q=0$, then we have~$p=M$. Without loss of generality,  we may assume $$\mu=[\xx_1^{n_1}\cdots\xx_{\rr}^{n_{\rr}};
 \xx_1^{m_1}\cdots\xx_{\rr}^{m_{\rr}}],$$ where~${\nn_1,\dots,\nn_{_{\MMM}}>\NNN_{\SS}}$, and there exists some integer~$t\in\{1,...,\rr\}$ such that $\mm_t\geq 1$.
 For every monomial~$\nu\in \Irr(\SS)$ and for every monomial~$\nu'\in \NF(\XXX)$ with~$\nu'|\nu$, it is clear that~$\nu'$ also lies in~$\Irr(\SS)$. So we deduce $$[\xx_{1}^{\NNN_{\SS}+1}\cdots\xx_{_{\MMM}}^{\NNN_{\SS}+1};
 \xx_t]\in\Irr(\SS).$$
 Now we show that
 $$\Irr(\SS)\supseteq \{[\xx_1^{l_1}\cdots\xx_{_{\MMM}}^{l_{_{\MMM}}};\xx_t]
 \in\NF(\XXX)\sm \XXX
 \mid l_{1},\dots,l_{_{\MMM}}\geq 0\}.$$
 It suffices to show that $[\xx_{1}^{\nn}\cdots\xx_{_{\MMM}}^{\nn};\xx_t]$ belongs to $\Irr(\SS)$ for every $\nn\geq\NNN_{\SS}+1$.
Suppose for the contrary that   $[\xx_{1}^{\nn}\cdots\xx_{_{\MMM}}^{\nn};\xx_t]\notin\Irr(\SS)$ for some~$\nn>\NNN_{\SS}+1$.
Then we define
$$\nn_0=\minn\{\nn\mid [\xx_{1}^{\nn}\cdots\xx_{_{\MMM}}^{\nn};\xx_t]\notin\Irr(\SS)\}.$$
Clearly, we have
$$\nn_0>\NNN_{\SS}+1$$
and $[\xx_{1}^{\nn_0-1}\cdots\xx_{_{\MMM}}^{\nn_0-1};\xx_t]\in\Irr(\SS)$.
Therefore,  there exists some element~$\ss\in\SS$ such that we have~$\bar\ss|[\xx_{1}^{\nn_0}\cdots\xx_{_{\MMM}}^{\nn_0};\xx_t]$ and $\bar\ss\nmid [\xx_{1}^{\nn_0-1}\cdots\xx_{_{\MMM}}^{\nn_0-1};\xx_t]$.
By the definition of~$\NNN_\SS$, we deduce~$\NNN_\SS\geq\nn_0$,
which contradicts the fact that~$\nn_0>\NNN_{\SS}+1$.
So we have
$$\{[\xx_1^{l_1}\cdots\xx_{_{\MMM}}^{l_{_{\MMM}}};\xx_t]
 \in\NF(\XXX)\sm \XXX
 \mid l_{1},\dots,l_{_{\MMM}}\geq 0\} \subseteq\Irr(\SS).$$
Therefore, if~$q=0$, then there exist~$M$ distinct letters~$x_{i_1}\wdots x_{i_M}\in X$, and~$x_t\in X$ such that
\begin{equation} \label{q0}
\{[\xx_{i_1}^{l_1}\cdots\xx_{i_{\MMM}}^{l_{_{\MMM}}};\xx_t]
 \in\NF(\XXX) \sm \XXX
 \mid l_{1},\dots,l_{_{\MMM}}\geq 0\} \subseteq \Irr(S).
 \end{equation}
 If $p=0$, then similarly there exist~$M$ distinct letters~$x_{i_1}\wdots x_{i_M}\in X$, and~$x_t\in X$ such that
 \begin{equation}\label{p0}
 \{[\xx_t;\xx_{i_1}^{l_1}\cdots\xx_{i_{\MMM}}^{l_{_{\MMM}}}]
\in\NF(\XXX) \sm \XXX
\mid l_{1},\dots,l_{_{\MMM}}\geq0\}\subseteq\Irr(\SS).
\end{equation}

  If~$p,q\geq 1$, then by a similar proof as for the case when~$q=0$, we can show that
\begin{equation}\label{pq}
\{[\xx_{\ii_1}^{\nn_{\ii_1}}\cdots\xx_{\ii_p}^{\nn_{\ii_p}};
\xx_{\jj_1}^{\mm_{\jj_1}}\cdots\xx_{\jj_q}^{\mm_{\jj_q}}]
\in\NF(\XXX)\sm \XXX \mid \nn_{\ii_1}\wdots \mm_{\jj_q}\geq 0\}\subseteq\Irr(\SS).
\end{equation}
Combining~\eqref{q0},\eqref{p0} and \eqref{pq}, we have~$\gkd{\BC(\XXX|\SS)}\geq \MMM$. Finally,  by~\eqref{b1},  we obtain~$\gkd{\BC(\XXX|\SS)}=\MMM$.
\end{proof}

\begin{theorem}\label{GKBC}
Let  $\XXX=\{\xx_1,\dots,\xx_{\rr}\}$ be a nonempty finite set and let~$\MB=\BC(\XXX|\SS)$, where~$\SS$ is a finite \gsb\ for $\BC(\XXX|S)$. Then we have
$\gkd{\mB}=\Sd(\{ \mu \mid \mu=[\xx_{1}^{\nn_{1}}\cdots\xx_{\rr}^{\nn_{\rr}};
\xx_{1}^{\mm_{1}}\cdots\xx_{\rr}^{\mm_{\rr}}]\in \Irr(\SS)\sm\XXX,  \nn_i,\mm_j\in\{0,1, \NNN_{\SS}+1\}, i,j\in\{1\wdots r\}\}).$
In particular,  the Gelfand-Kirillov dimension of a finitely generated bicommutative algebra is a nonnegative integer.
\end{theorem}
\begin{proof}
Define~$$\MMM'=\Sd(\{ \mu \mid \mu=[\xx_{1}^{\nn_{1}}\cdots\xx_{\rr}^{\nn_{\rr}};
\xx_{1}^{\mm_{1}}\cdots\xx_{\rr}^{\mm_{\rr}}]\in \Irr(\SS)\sm\XXX,  \nn_i,\mm_j\in\{0,1, \NNN_{\SS}+1\}, 1\leq i,j\leq r\})$$
and
$$\MMM=\Sd(\Irr(\SS)).$$
Clearly we have~$\MMM'\leq\MMM$. By Lemma~\ref{lemma-GKBC},  it suffices to show~$\MMM\leq \MMM'$.

If $\MMM'=0$, then we have
 $$\Irr(\SS)\subseteq \XXX\cup \{[\xx_{1}^{\pp_{1}}\cdots\xx_{\rr}^{\pp_{\rr}};
\xx_{1}^{\qq_{1}}\cdots\xx_{\rr}^{\qq_{\rr}}] \in \NF(\XXX)\sm \XXX \mid \pp_{i},\qq_{j}\leq \NNN_{\SS}, 1\leq i,j\leq r\}.$$
Then it follows immediately that we have~$M=\gkd{\BC(\XXX|\SS)}=0=M'$.

If $\MMM'>0$, then we have~$M\geq M'>0$, and thus  there exists some monomial~$\nu=[\xx_{1}^{\nn_{1}}\cdots\xx_{\rr}^{\nn_{\rr}};
\xx_{1}^{\mm_{1}}\cdots\xx_{\rr}^{\mm_{\rr}}]\in\Irr(\SS)\sm \XXX$ satisfying~$\Sd(\nu)=\MMM>0$. By the construction of~$\Irr(S)$, every monomial~$\nu'$ satisfying~$\nu'|\nu$ lies in~$\Irr(S)$.
 For every~$i\in\{1\wdots r\}$, we define
\begin{eqnarray*}
n'_i=
\begin{cases}
N_S+1,& \mbox{if } n_i>N_S,\\
1,& \mbox{if } n_i\in\{1\wdots N_S\},\\
0,& \mbox{if } n_i=0;
\end{cases}
\ \ \ \ \ \ \ \ \
m'_i=
\begin{cases}
N_S+1,& \mbox{if } m_i>N_S,\\
1,& \mbox{if } m_i\in\{1\wdots N_S\},\\
0,& \mbox{if } m_i=0.
\end{cases}
\end{eqnarray*}
Then it follows that we have
$$\Sd(\nu')=M$$
and
$$\nu'=[\xx_{1}^{\nn_{1}'}\cdots\xx_{\rr}^{\nn_{\rr}'};
\xx_{1}^{\mm_{1}'}\cdots\xx_{\rr}^{\mm_{\rr}'}]\in B,$$
where  $$
 B:=\{ \mu \mid \mu=[\xx_{1}^{\nn_{1}}\cdots\xx_{\rr}^{\nn_{\rr}};
\xx_{1}^{\mm_{1}}\cdots\xx_{\rr}^{\mm_{\rr}}]\in \Irr(\SS)\sm\XXX,  \nn_i,\mm_j\in\{0,1, \NNN_{\SS}+1\}, i,j\in\{1\wdots r\}\}.
 $$
Therefore we have
 $\MMM'\geq\Sd(\nu)=\MMM$.
 The proof is completed.
\end{proof}
 Theorem~\ref{GKBC} offers a fast algorithm for calculating the Gelfand-Kirillov dimension of a finitely presented bicommutative algebra by investigating the finite set
$$\{ \mu\mid \mu=[\xx_{1}^{\nn_{1}}\cdots\xx_{\rr}^{\nn_{\rr}};
\xx_{1}^{\mm_{1}}\cdots\xx_{\rr}^{\mm_{\rr}}]\in\Irr(\SS)\setminus X,  \nn_i,\mm_j\in\{0,1, \NNN_{\SS}+1\}, i,j\in\{1\wdots r\}\}$$
rather than~$\Irr(\SS)$.

\begin{exam}
Let  $\XXX=\{\xx_1,\dots,\xx_{\rr}\}$ be a nonempty finite set. Then  the Gelfand-Kirillov dimension of the free bicommutative algebra $\BC(X)$ generated by $X$ is $2r$.
\end{exam}
\begin{proof}
As a special case of Theorem~\ref{GKBC}, say $\XXX=\{\xx_1,\dots,\xx_{\rr}\}$ and $\SS=\emptyset$, then we
have~$\NNN_S=1$ and $[\xx_{1}^{2}\cdots\xx_{\rr}^{2};
\xx_{1}^{2}\cdots\xx_{\rr}^{2}]\in \Irr(\SS)$, and thus $\gkd{\BC(X)}=2r$.
\end{proof}

\begin{coro}\label{biporp-1}
Let  $\XXX=\{\xx_1,\dots,\xx_{\rr}\}$ be a nonempty finite set.
Let $\SS$ be a nonempty finite subset of~$\NF(\XXX)\sm\XXX$ and~$\mu\in\NF(\XXX)\sm\XXX$. Suppose that~$\gkd{\BC(\XXX|\SS)}=n\geq 3$.
Let
$$B=\{ \mu'\mid \mu'=[\xx_{1}^{\nn_{1}}\cdots\xx_{\rr}^{\nn_{\rr}};
\xx_{1}^{\mm_{1}}\cdots\xx_{\rr}^{\mm_{\rr}}]\in\NF(\XXX)\setminus X,  \nn_i,\mm_j\in\{0,1, \NNN_{\SS}+1\}, i,j\in\{1\wdots r\}\}.$$
If there exists some monomial~$\nu\in B\cap \Irr(\SS\cup\{\mu\})$ with~$\Sd(\nu)=n$, then we have
$$
\gkd{\BC(\XXX|\SS\cup\{\mu\})}=\gkd{\BC(\XXX|\SS)};
$$
Otherwise, we have
$$
\gkd{\BC(\XXX|\SS\cup\{\mu\})}=\gkd{\BC(\XXX|\SS)}-1.
$$
\end{coro}
\begin{proof}
Obviously, every nonempty subset of $\NF(\XXX)\sm\XXX$ is a \gsb\ in $\BC(\XXX)$ and thus~$S$ is a \gsb\ for~$\BC(\XXX|S)$.
 Then we have
 $$\gkd{\BC(\XXX|\SS)}=\Sd(B\cap\Irr(\SS))=n$$ for some positive integer~$n\geq 3$.

If there exists some monomial~$\nu\in B\cap \Irr(\SS\cup\{\mu\})$ with~$\Sd(\nu)=n$, then we have
$$
\gkd{\BC(\XXX|\SS\cup\{\mu\})}=n=\gkd{\BC(\XXX|\SS)}.
$$
If for every monomial~$\nu'\in B\cap \Irr(S)$ with~$\Sd(\nu')=n$, $\nu'$ does not lie in~$\Irr(\SS\cup\{\mu\})$,
then we have~$\mu|\nu'$.
Suppose~$\nu_0\in B\cap \Irr(S)$ and~$\Sd(\nu_0)=n$.
Then we may assume
$$\nu_0=[\xx_1^{n_1}\cdots\xx_{\rr}^{n_{\rr}};
\xx_1^{m_1}\cdots\xx_{\rr}^{m_{\rr}}]\in\Irr(\SS)\sm\XXX, $$ where~$\nn_{\ii_1}=\cdots=\nn_{\ii_p}=\mm_{\jj_1}=\cdots
=\mm_{\jj_q}=\NNN_{\SS}+1$,
$p+q=n$, $ p,q\in\mathbb{N}$, $\nn_{l},\mm_{l'} \in \{0,1\}$ for all~$l\in \{1\wdots r\}\setminus\{i_1\wdots i_p\} $ and~$l'\in \{1\wdots r\}\setminus\{j_1\wdots j_q\}$.
Since~$n\geq 3$, we have~$p>1$ or~$q>1$. Without loss of generality, we assume $p>1$ and suppose~$\ii_1=1$, $\ii_2=2$ and~$[\xx_{1};x_t]|\mu$ for some~$x_t\in X$. Then $[\xx_2^{n_2}\cdots\xx_{\rr}^{n_{\rr}};
\xx_1^{m_1}\cdots\xx_{\rr}^{m_{\rr}}]\in\Irr(\SS\cup\{\mu\})\sm\XXX$. Therefore, we have
$$
\gkd{\BC(\XXX|\SS\cup\{\mu\})}=n-1=\gkd{\BC(\XXX|\SS)}-1.
$$
The proof is completed.
\end{proof}

If we have~$\gkd{\BC(X|S)}=n>0$, then in general we cannot have a good estimation on~$\gkd{\BC(X|S\cup\{x\})}$, where~$x$ lies in~$X$. For instance, let~$X=\{x_1\wdots x_r\}\ (r>1)$ and let~$S=\{[x;y]\mid x,y\in X, x\neq x_1 \}$. Then we have~$N_S=1$ and $[x_1^{2};x_1^{2}...x_r^{2}]\in \Irr(S)$. Clearly, for every~$\mu\in \Irr(S)$, we have~$\Sd(\mu)\leq r+1$, so~$\gkd{\BC(X|S)}=r+1$. Moreover, we have~$\gkd{\BC(X|S\cup\{x_1\})}=0$ and~$\gkd{\BC(X|S\cup\{x_2\})}=r$.

The following example shows that if~$\gkd{\BC(\XXX|\SS)}=2$, then we cannot obtain a similar result as in Corollary~\ref{biporp-1}.
\begin{exam}
Let  $\XXX=\{\xx_1,\xx_2\}$ and
$\SS=\{[\xx_1;\xx_1],[\xx_1;\xx_2],[\xx_2;\xx_1]\}$.
Then we obtain that~$\NNN_S=1$.
Let~$B:=\{ \mu\mid \mu=[\xx_{1}^{\nn_{1}}\xx_{2}^{\nn_{2}};
\xx_{1}^{\mm_{1}}\xx_{2}^{\mm_{2}}]\in\NF(\XXX)\sm\XXX,  \nn_1,\nn_2,\mm_1,\mm_2\in\{0,1, 2\}\}$.
Clearly $\SS$ is a \gsb\ for $\BC(\XXX|\SS)$ and we have~$B\cap\Irr(\SS)=\{[\xx_2^2;\xx_2^2]\}$. Thus~$\gkd{\BC(\XXX|\SS)}=2$.
Let~$\nu=[\xx_2;\xx_2]$.
Then we have $B\cap\Irr(\SS\cup\{\nu\})=\emptyset$. Therefore
$$\gkd{\BC(\XXX|\SS\cup\{\nu\})}=0=\gkd{\BC(\XXX|\SS)}-2.$$
\end{exam}

The following proposition offers a rough estimation of the Gelfand-Kirillov dimension of a finitely presented bicommutative algebra.
\begin{prop}\label{GKBC2}
Let~$\XXX=\{\xx_1,\dots,\xx_{\rr}\}$ be a nonempty finite set and~$S=\{\ff_1\wdots\ff_t\}$ a finite \gsb\ for~$\BC(\XXX|\SS)$, where~$r>1$.
Then the following hold:

\ITEM1 If~$1\leq t\leq r$ and $\ov{\SS}\cap\XXX=\emptyset$, then we have
$$\gkd{\BC(\XXX|\SS)}\in\{ 2r-t\wdots 2r-1\};$$

\ITEM2 If~$t> r$, then we have
$$\gkd{\BC(\XXX|\SS)}\in \{0\wdots 2r-1\}.$$
\end{prop}

\begin{proof}
\ITEM1
We use induction on~$t$.
If $t=1$, then we may assume~$\ff_1=[x_1^{\nn_1}\cdots x_r^{\nn_r};x_1^{\mm_1}\cdots x_r^{\mm_r}]$, where~$\nn_1>0$ and~$ \mm_i>0$ for some integer~$i\leq r$. Thus we obtain that
$$[x_2^{\NNN_S+1}\cdots x_r^{\NNN_S+1};x_1^{\NNN_S+1}\cdots x_r^{\NNN_S+1}]\in \Irr(S) \mbox{ and } [x_1^{\NNN_S+1}\cdots x_r^{\NNN_S+1};x_1^{\NNN_S+1}\cdots x_r^{\NNN_S+1}]\notin \Irr(S). $$
  Therefore by Theorem~\ref{GKBC}, we
have
$$\gkd{\BC(\XXX|\SS)}=2r-1.$$
Now we assume~$t>1$.
Obviously, $\{\ov{\ff_1}\wdots\ov{\ff_{t-1}}\}$ is a \gsb\ in $\BC(\XXX)$ and $\Irr(\SS)=\Irr(\ov{\SS})$.
By induction hypothesis, we have   $$p:=\gkd{\BC(\XXX|\{\ov{\ff_1}\wdots\ov{\ff_{t-1}}\})}\in\{ 2r-t+1 \wdots 2r-1\}.$$
Since~$2\leq t\leq r$, we have $2r-t+1\geq3$.
So by Corollary~\ref{biporp-1}, we have
$$\gkd{\BC(\XXX|\SS)}=\gkd{\BC(\XXX|\ov{\SS})}\in \{p-1,p\}\subseteq \{ 2r-t\wdots 2r-1\}.$$

\ITEM2
Since~$[x_1^{\NNN_S+1}\cdots x_r^{\NNN_S+1};x_1^{\NNN_S+1}\cdots x_r^{\NNN_S+1}]\notin \Irr(S)$, by Theorem \ref{GKBC}, the result follows immediately.
\end{proof}

\begin{exam} Let~$\XXX=\{\xx_1,\dots,\xx_{\rr}\}$, $r>1$ and let~$t$ be a positive integer.
For every positive integer~$\pp\leq\min \{r,t\}$,
we define $m=t-p$ and define
$$\SS=\{[\xx_i;\xx_i]\mid1\leq i\leq p\}\cup\{[\xx_{r}^{i};\xx_1^{3m-i}]\mid m\leq i\leq 2m-1\}.$$
Let~$B:=\{ [x_1^{\nn_1}\cdots x_r^{\nn_r};x_1^{\mm_1}\cdots x_r^{\mm_r}]\in\NF(\XXX)\sm\XXX \mid  \nn_i,\mm_j\in\{0,1, \NNN_S+1\}, i,j\in\{1\wdots r\}\}$.
Clearly~$\SS$ is a \gsb\ for $\BC(\XXX|\SS)$ with~$\#\SS=t$ and we have~$N_S=\max\{2m, 1\}$. Then we obtain
$$[\xx_{1}^{\NNN_S+1}\xx_{p+1}^{\NNN_S+1}\cdots\xx_r^{\NNN_S+1};
\xx_2^{\NNN_S+1}\cdots\xx_{r}^{\NNN_S+1}]\in B\cap\Irr(\SS),$$
and for every monomial~$\mu\in B$ with $\Sd(\mu)=2r-p+1$, $\mu$ does not lie in~$\Irr(\SS)$.
Thus we have
$$
\gkd{\BC(\XXX|\SS)}=2r-p.
$$
\end{exam}
\subsection*{Acknowledgment}
The authors are grateful to L. A. Bokut and Yu Li for bringing the topic of bicommutative algebras to our attention.

\subsection*{Disclosure statement}
No potential conflict of interest was reported by the authors.

\subsection*{Funding}
Y.C. is supported by the NNSF of China (11571121, 12071156); Z.Z. is supported by the fellowship of China Postdoctoral Science Foundation 2021M691099 and the Young Teacher Research and Cultivation  Foundation of South China Normal University 20KJ02.

\end{document}